\documentclass[10pt]{article}
\usepackage{amsmath}
\usepackage{amssymb}
\usepackage{amsthm}
\usepackage[compress,noadjust]{cite}
\numberwithin{equation}{section}

\begin{document}
\title{Weighted estimates for vector-valued intrinsic square functions and commutators in the Morrey type spaces}
\author{Hua Wang \footnote{E-mail address: wanghua@pku.edu.cn.}\\
\footnotesize{College of Mathematics and Econometrics, Hunan University, Changsha 410082, P. R. China}}
\date{}
\maketitle

\begin{abstract}
In this paper, the boundedness properties of vector-valued intrinsic square functions and their vector-valued commutators with $BMO(\mathbb R^n)$ functions are discussed. We first show the weighted strong type and weak type estimates of vector-valued intrinsic square functions in the Morrey type spaces. Then we obtain weighted strong type estimates of vector-valued analogues of commutators in Morrey type spaces. In the endpoint case, we establish the weighted weak $L\log L$-type estimates for these vector-valued commutators in the setting of weighted Lebesgue spaces. Furthermore, we prove weighted endpoint estimates of these commutator operators in Morrey type spaces. In particular, we can obtain strong type and endpoint estimates of vector-valued intrinsic square functions and their commutators in the weighted Morrey spaces and the generalized Morrey spaces.
\end{abstract}

\section{Introduction}

The intrinsic square functions were first introduced by Wilson in \cite{wilson1,wilson2}; they are defined as follows. For $0<\alpha\le1$, let ${\mathcal C}_\alpha$ be the family of functions $\varphi:\mathbb R^n\longmapsto\mathbb R$ such that $\varphi$'s support is contained in $\{x\in\mathbb R^n: |x|\le1\}$, $\int_{\mathbb R^n}\varphi(x)\,dx=0$, and for all $x, x'\in \mathbb R^n$,
\begin{equation*}
\big|\varphi(x)-\varphi(x')\big|\leq \big|x-x'\big|^{\alpha}.
\end{equation*}
For $(y,t)\in {\mathbb R}^{n+1}_{+}=\mathbb R^n\times(0,+\infty)$ and $f\in L^1_{{loc}}(\mathbb R^n)$, we set
\begin{equation*}
A_\alpha(f)(y,t)=\sup_{\varphi\in{\mathcal C}_\alpha}\big|\varphi_t*f(y)\big|=\sup_{\varphi\in{\mathcal C}_\alpha}\bigg|\int_{\mathbb R^n}\varphi_t(y-z)f(z)\,dz\bigg|,
\end{equation*}
where $\varphi_t$ denotes the usual $L^1$ dilation of $\varphi:\varphi_t(y)=t^{-n}\varphi(y/t)$. Then we define the intrinsic square function of $f$ (of order $\alpha$) by the formula
\begin{equation}
\mathcal S_{\alpha}(f)(x):=\left(\iint_{\Gamma(x)}\Big(A_\alpha(f)(y,t)\Big)^2\frac{dydt}{t^{n+1}}\right)^{1/2},
\end{equation}
where $\Gamma(x)$ denotes the usual cone of aperture one:
\begin{equation*}
\Gamma(x):=\big\{(y,t)\in{\mathbb R}^{n+1}_+:|x-y|<t\big\}.
\end{equation*}
This new function is independent of any particular kernel, and it dominates pointwise the classical square function~(Lusin area integral) and its real-variable generalizations, one can see more details in \cite{wilson1,wilson2}. Let $\vec{f}=(f_1,f_2,\ldots)$ be a sequence of locally integrable functions on $\mathbb R^n$. For any $x\in\mathbb R^n$, Wilson \cite{wilson2} also defined the vector-valued intrinsic square functions of $\vec{f}$ by
\begin{equation}\label{vectorvalued}
\mathcal S_\alpha(\vec{f})(x):=\bigg(\sum_{j=1}^\infty\big|\mathcal S_\alpha(f_j)(x)\big|^2\bigg)^{1/2}.
\end{equation}

In \cite{wilson2}, Wilson has established the following two results.

\newtheorem*{thma}{Theorem A}

\begin{thma}[\cite{wilson2}]
Let $0<\alpha\leq 1$, $1<p<\infty$ and $w\in A_p$(\mbox{Muckenhoupt weight class}). Then there exists a constant $C>0$ independent of $\vec{f}=(f_1,f_2,\ldots)$ such that
\begin{equation*}
\bigg\|\bigg(\sum_{j}\big|\mathcal S_\alpha(f_j)\big|^2\bigg)^{1/2}\bigg\|_{L^p_w}
\leq C \bigg\|\bigg(\sum_{j}\big|f_j\big|^2\bigg)^{1/2}\bigg\|_{L^p_w}.
\end{equation*}
\end{thma}

\newtheorem*{thmb'}{Theorem \r{B}}

\begin{thmb'}[\cite{wilson2}]
Let $0<\alpha\leq 1$ and $p=1$. Then for any given weight function $w$ and $\lambda>0$, there exists a constant $C>0$ independent of $\vec{f}=(f_1,f_2,\ldots)$ and $\lambda>0$ such that
\begin{equation*}
w\bigg(\bigg\{x\in\mathbb R^n:\bigg(\sum_{j}\big|\mathcal S_{\alpha}(f_j)(x)\big|^2\bigg)^{1/2}>\lambda\bigg\}\bigg)
\leq \frac{C}{\lambda}\int_{\mathbb R^n}\bigg(\sum_{j}\big|f_j(x)\big|^2\bigg)^{1/2}M(w)(x)\,dx,
\end{equation*}
where $M$ denotes the standard Hardy--Littlewood maximal operator.
\end{thmb'}

If we take $w\in A_1$, then $M(w)(x)\le C\cdot w(x)$ for a.e.$x\in\mathbb R^n$ by the definition of $A_1$ weight (see Section 2). Hence, as a straightforward consequence of Theorem \r{B}, we obtain

\newtheorem*{thmb}{Theorem B}

\begin{thmb}
Let $0<\alpha\leq 1$, $p=1$ and $w\in A_1$. Then there exists a constant $C>0$ independent of $\vec{f}=(f_1,f_2,\ldots)$ such that
\begin{equation*}
\bigg\|\bigg(\sum_{j}\big|\mathcal S_{\alpha}(f_j)\big|^2\bigg)^{1/2}\bigg\|_{WL^1_w}
\leq C \bigg\|\bigg(\sum_{j}\big|f_j\big|^2\bigg)^{1/2}\bigg\|_{L^1_w}.
\end{equation*}
\end{thmb}

Let $b$ be a locally integrable function on $\mathbb R^n$ and $0<\alpha\leq 1$, the commutators generated by $b$ and intrinsic square functions $\mathcal S_{\alpha}$ are defined by the following expression in \cite{wang1}.
\begin{equation}
\big[b,\mathcal S_\alpha\big](f)(x):=\left(\iint_{\Gamma(x)}\sup_{\varphi\in{\mathcal C}_\alpha}\bigg|\int_{\mathbb R^n}\big[b(x)-b(z)\big]\varphi_t(y-z)f(z)\,dz\bigg|^2\frac{dydt}{t^{n+1}}\right)^{1/2}.
\end{equation}

In this paper, we will consider the vector-valued analogues of these commutator operators. Let $\vec{f}=(f_1,f_2,\ldots)$ be a sequence of locally integrable functions on $\mathbb R^n$. For any $x\in\mathbb R^n$, in the same way, we can define the commutators for vector-valued intrinsic square functions of $\vec{f}$ as
\begin{equation}\label{vectorvaluedc}
\big[b,\mathcal S_\alpha\big](\vec{f})(x):=\bigg(\sum_{j=1}^\infty\big|\big[b,\mathcal S_\alpha\big](f_j)(x)\big|^2\bigg)^{1/2}.
\end{equation}

On the other hand, the classical Morrey space was originally introduced by Morrey in \cite{morrey} to study the local behavior of solutions to second order elliptic partial differential equations. Since then, this space played an important role in studying the regularity of solutions to partial differential equations. In \cite{mizuhara}, Mizuhara introduced the generalized Morrey space $L^{p,\Theta}(\mathbb R^n)$ which was later extended and studied by many authors. In \cite{komori}, Komori and Shirai defined a version of the weighted Morrey space $L^{p,\kappa}(w)$ which is a natural generalization of weighted Lebesgue space. Recently, in \cite{wang2}, we have established the strong type and weak type estimates for vector-valued intrinsic square functions on $L^{p,\Theta}(\mathbb R^n)$ and $L^{p,\kappa}(w)$.

The main purpose of this paper is twofold. We first define a new kind of Morrey type spaces $\mathcal M^{p,\theta}(w)$ containing generalized Morrey space $L^{p,\Theta}(\mathbb R^n)$ and weighted Morrey space $L^{p,\kappa}(w)$ as special cases, and then we will discuss the boundedness properties of vector-valued intrinsic square functions \eqref{vectorvalued} and vector-valued commutators \eqref{vectorvaluedc} defined above in these Morrey type spaces $\mathcal M^{p,\theta}(w)$ for all $1\leq p<\infty$.

\section{Main results}

\subsection{Notations and preliminaries}

A weight $w$ will always mean a positive function which is locally integrable on $\mathbb R^n$, $B=B(x_0,r_B)=\{x\in\mathbb R^n:|x-x_0|<r_B\}$ denotes the open ball centered at $x_0$ and with radius $r_B>0$. For $1<p<\infty$, a weight function $w$ is said to belong to the Muckenhoupt's class $A_p$, if there is a constant $C>0$ such that for every ball $B\subseteq \mathbb R^n$~(see \cite{garcia,muckenhoupt}),
\begin{equation*}
\left(\frac1{|B|}\int_B w(x)\,dx\right)^{1/p}\left(\frac1{|B|}\int_B w(x)^{-p'/p}\,dx\right)^{1/{p'}}\leq C,
\end{equation*}
where $p'$ is the dual of $p$ such that $1/p+1/{p'}=1$. For the case $p=1$, $w\in A_1$, if there is a constant $C>0$ such that for every ball $B\subseteq \mathbb R^n$,
\begin{equation*}
\frac1{|B|}\int_B w(x)\,dx\le C\cdot\underset{x\in B}{\mbox{ess\,inf}}\;w(x).
\end{equation*}
We also define $A_\infty=\bigcup_{1\leq p<\infty}A_p$. It is well known that if $w\in A_p$ with $1\leq p<\infty$, then for any ball $B$, there exists an absolute constant $C>0$ such that
\begin{equation}\label{weights}
w(2B)\le C\cdot w(B).
\end{equation}
In general, for $w\in A_1$ and any $j\in\mathbb Z_+$, there exists an absolute constant $C>0$ such that (see \cite{garcia})
\begin{equation}\label{general weights}
w\big(2^j B\big)\le C\cdot 2^{jn}w(B).
\end{equation}
Moreover, if $w\in A_\infty$, then for all balls $B$ and all measurable subsets $E$ of $B$, there exists a number $\delta>0$ independent of $E$ and $B$ such that (see \cite{garcia})
\begin{equation}\label{compare}
\frac{w(E)}{w(B)}\le C\left(\frac{|E|}{|B|}\right)^\delta.
\end{equation}
A weight function $w$ is said to belong to the reverse H\"{o}lder class $RH_r$, if there exist two constants $r>1$ and
$C>0$ such that the following reverse H\"{o}lder inequality holds for every
ball $B\subseteq \mathbb R^n$.
\begin{equation*}
\left(\frac{1}{|B|}\int_B w(x)^r\,dx\right)^{1/r}\le C\left(\frac{1}{|B|}\int_B w(x)\,dx\right).
\end{equation*}
Given a ball $B$ and $\lambda>0$, $\lambda B$ denotes the ball with the same center as $B$ whose radius is $\lambda$ times that of $B$. For a given weight function $w$ and a measurable set $E$, we also denote the Lebesgue measure of $E$ by $|E|$ and the weighted measure of $E$ by $w(E)$, where $w(E)=\int_E w(x)\,dx$. Equivalently, we could define the above notions with cubes
instead of balls. Hence we shall use these two different definitions appropriate to calculations.

Given a weight function $w$ on $\mathbb R^n$, for $1\leq p<\infty$, the weighted Lebesgue space $L^p_w(\mathbb R^n)$ is defined as the set of all functions $f$ such that
\begin{equation*}
\big\|f\big\|_{L^p_w}:=\bigg(\int_{\mathbb R^n}|f(x)|^pw(x)\,dx\bigg)^{1/p}<\infty.
\end{equation*}
We also denote by $WL^1_w(\mathbb R^n)$ the weighted weak Lebesgue space consisting of all measurable functions $f$ such that
\begin{equation*}
\big\|f\big\|_{WL^1_w}:=
\sup_{\lambda>0}\lambda\cdot w\big(\big\{x\in\mathbb R^n:|f(x)|>\lambda \big\}\big)<\infty.
\end{equation*}

We next recall some basic definitions and facts about Orlicz spaces needed for the proofs of the main results. For more information on the subject, one can see \cite{rao}. A function $\Phi$ is called a Young function if it is continuous, nonnegative, convex and strictly increasing on $[0,+\infty)$ with $\Phi(0)=0$ and $\Phi(t)\to +\infty$ as $t\to +\infty$. We define the $\Phi$-average of a function $f$ over a ball $B$ by means of the following Luxemburg norm:
\begin{equation*}
\big\|f\big\|_{\Phi,B}:=\inf\left\{\sigma>0:\frac{1}{|B|}\int_B\Phi\left(\frac{|f(x)|}{\sigma}\right)dx\leq1\right\}.
\end{equation*}
An equivalent norm that is often useful in calculations is as follows(see \cite{rao,perez1}):
\begin{equation}\label{equiv norm}
\big\|f\big\|_{\Phi,B}\leq \inf_{\eta>0}\left\{\eta+\frac{\eta}{|B|}\int_B\Phi\left(\frac{|f(x)|}{\eta}\right)dx\right\}\leq 2\big\|f\big\|_{\Phi,B}.
\end{equation}
Given a Young function $\Phi$, we use $\bar\Phi$ to denote the complementary Young function associated to $\Phi$. Then the following generalized H\"older's inequality holds for any given ball $B$ (see \cite{perez1,perez2}).
\begin{equation*}
\frac{1}{|B|}\int_B|f(x)\cdot g(x)|\,dx\leq 2\big\|f\big\|_{\Phi,B}\big\|g\big\|_{\bar\Phi,B}.
\end{equation*}
In order to deal with the weighted case, for $w\in A_\infty$, we need to define the weighted $\Phi$-average of a function $f$ over a ball $B$ by means of the weighted Luxemburg norm:
\begin{equation*}
\big\|f\big\|_{\Phi(w),B}:=\inf\left\{\sigma>0:\frac{1}{w(B)}\int_B\Phi\left(\frac{|f(x)|}{\sigma}\right)w(x)\,dx\leq1\right\}.
\end{equation*}
It can be shown that for $w\in A_\infty$(see \cite{rao,zhang}),
\begin{equation}\label{equiv norm with weight}
\big\|f\big\|_{\Phi(w),B}\approx \inf_{\eta>0}\left\{\eta+\frac{\eta}{w(B)}\int_B\Phi\left(\frac{|f(x)|}{\eta}\right)w(x)\,dx\right\},
\end{equation}
and
\begin{equation*}
\frac{1}{w(B)}\int_B|f(x)\cdot g(x)|w(x)\,dx\leq C\big\|f\big\|_{\Phi(w),B}\big\|g\big\|_{\bar\Phi(w),B}.
\end{equation*}
The young function that we are going to use is $\Phi(t)=t\cdot(1+\log^+t)$ with its complementary Young function $\bar\Phi(t)\approx\exp(t)$. Here by $A\approx B$, we mean that there exists a constant $C>1$ such that $\frac1C\le\frac AB\le C$. In the present situation, we denote
\begin{equation*}
\big\|f\big\|_{L\log L(w),B}=\big\|f\big\|_{\Phi(w),B}, \qquad \big\|g\big\|_{\exp L(w),B}=\big\|g\big\|_{\bar\Phi(w),B}.
\end{equation*}
By the generalized H\"older's inequality with weight, we have (see \cite{perez1,zhang})
\begin{equation}\label{weighted holder}
\frac{1}{w(B)}\int_B|f(x)\cdot g(x)|w(x)\,dx\leq C\big\|f\big\|_{L\log L(w),B}\big\|g\big\|_{\exp L(w),B}.
\end{equation}

Let us now recall the definition of the space of $BMO(\mathbb R^n)$ (Bounded Mean Oscillation) (see \cite{duoand,john}).
A locally integrable function $b$ is said to be in $BMO(\mathbb R^n)$, if
\begin{equation*}
\|b\|_*:=\sup_{B}\frac{1}{|B|}\int_B|b(x)-b_B|\,dx<\infty,
\end{equation*}
where $b_B$ stands for the average of $b$ on $B$, i.e., $b_B=\frac{1}{|B|}\int_B b(y)\,dy$ and the supremum is taken
over all balls $B$ in $\mathbb R^n$. Modulo constants, the space $BMO(\mathbb R^n)$ is a Banach space with respect to the norm $\|\cdot\|_*$.
By the John--Nirenberg's inequality, it is not difficult to see that for any $w\in A_\infty$ and any given ball $B$ (see \cite{zhang}),
\begin{equation}\label{weighted exp}
\big\|b-b_B\big\|_{\exp L(w),B}\leq C\|b\|_*.
\end{equation}

\subsection{Morrey type spaces}

\newtheorem{defn}{Definition}[section]

\begin{defn}[\cite{komori}]
Let $1\leq p<\infty$, $0<\kappa<1$ and $w$ be a weight function on $\mathbb R^n$. Then the weighted Morrey space $L^{p,\kappa}(w)$ is defined by
\begin{equation*}
L^{p,\kappa}(w):=\left\{f\in L^p_{loc}(w):\big\|f\big\|_{L^{p,\kappa}(w)}
=\sup_B\left(\frac{1}{w(B)^{\kappa}}\int_B|f(x)|^pw(x)\,dx\right)^{1/p}<\infty\right\},
\end{equation*}
where the supremum is taken over all balls $B$ in $\mathbb R^n$. We also denote by $WL^{1,\kappa}(w)$ the weighted weak Morrey space of all measurable functions $f$ such that
\begin{equation*}
\sup_B\sup_{\lambda>0}\frac{1}{w(B)^{\kappa}}\lambda\cdot w\big(\big\{x\in B:|f(x)|>\lambda\big\}\big)
\leq C<\infty.
\end{equation*}
\end{defn}

Let $\Theta=\Theta(r)$, $r>0$, be a growth function, that is, a positive increasing function in $(0,+\infty)$ and satisfy the following doubling condition:
\begin{equation}\label{doubling}
\Theta(2r)\leq D\cdot\Theta(r), \quad \mbox{for all }\,r>0,
\end{equation}
where $D=D(\Theta)\ge1$ is a doubling constant independent of $r$.

\begin{defn}[\cite{mizuhara}]
Let $1\leq p<\infty$ and $\Theta$ be a growth function in $(0,+\infty)$. Then the generalized Morrey space $L^{p,\Theta}(\mathbb R^n)$ is defined as the set of all locally integrable functions $f$ for which
\begin{equation*}
\big\|f\big\|_{L^{p,\Theta}}:=\sup_{r>0;B(x_0,r)}\left(\frac{1}{\Theta(r)}\int_{B(x_0,r)}|f(x)|^p\,dx\right)^{1/p}<\infty,
\end{equation*}
where the supremum is taken over all balls $B(x_0,r)$ in $\mathbb R^n$ with $x_0\in\mathbb R^n$. We also denote by $WL^{1,\Theta}(\mathbb R^n)$ the generalized weak Morrey space of all measurable functions $f$ for which
\begin{equation*}
\sup_{B(x_0,r)}\sup_{\lambda>0}\frac{1}{\Theta(r)}\lambda\cdot\big|\big\{x\in B(x_0,r):|f(x)|>\lambda\big\}\big|
\leq C<\infty.
\end{equation*}
\end{defn}

In order to unify these two definitions, we will introduce Morrey type spaces as follows.
Let $0\leq\kappa<1$. Assume that $\theta(\cdot)$ is a positive increasing function defined in $(0,+\infty)$ and satisfies the following $\mathcal D_\kappa$ condition:
\begin{equation}\label{D condition}
\frac{\theta(\xi)}{\xi^\kappa}\leq C\cdot\frac{\theta(\xi')}{(\xi')^\kappa},\quad \mbox{for any}\;0<\xi'<\xi<+\infty,
\end{equation}
where $C>0$ is a constant independent of $\xi$ and $\xi'$.

\begin{defn}
Let $1\leq p<\infty$, $0\leq\kappa<1$, $\theta$ satisfy the $\mathcal D_\kappa$ condition $(\ref{D condition})$ and $w$ be a weight function on $\mathbb R^n$. We denote by $\mathcal M^{p,\theta}(w)$ the generalized weighted Morrey space of all locally integrable functions $f$ defined on $\mathbb R^n$, such that for every ball $B$ in $\mathbb R^n$,
\begin{equation*}
\left(\frac{1}{\theta(w(B))}\int_B |f(x)|^pw(x)\,dx\right)^{1/p}\leq C<\infty.
\end{equation*}
Then we let $\|f\|_{\mathcal M^{p,\theta}(w)}$ be the smallest constant $C>0$ satisfying the above estimate and $\mathcal M^{p,\theta}(w)$ becomes a Banach function space with norm $\|\cdot\|_{\mathcal M^{p,\theta}(w)}$. In the unweighted case(when $w$ equals a constant function), we denote the generalized unweighted Morrey space by $\mathcal M^{p,\theta}(\mathbb R^n)$. That is, let $1\leq p<\infty$ and $\theta$ satisfy the $\mathcal D_\kappa$ condition $(\ref{D condition})$ with $0\leq\kappa<1$, we define
\begin{equation*}
\mathcal M^{p,\theta}(\mathbb R^n):=\left\{f\in L^p_{loc}(\mathbb R^n):
\big\|f\big\|_{\mathcal M^{p,\theta}}=\sup_B\left(\frac{1}{\theta(|B|)}\int_B |f(x)|^p\,dx\right)^{1/p}<\infty\right\}.
\end{equation*}
\end{defn}

\begin{defn}
Let $p=1$, $0\leq\kappa<1$, $\theta$ satisfy the $\mathcal D_\kappa$ condition $(\ref{D condition})$ and $w$ be a weight function on $\mathbb R^n$. We denote by $W\mathcal M^{1,\theta}(w)$ the generalized weighted weak Morrey space consisting of all measurable functions $f$ defined on $\mathbb R^n$ for which
\begin{equation*}
\big\|f\big\|_{W\mathcal M^{1,\theta}(w)}:=\sup_B\sup_{\sigma>0}\frac{1}{\theta(w(B))}\sigma\cdot w\big(\big\{x\in B:|f(x)|>\sigma\big\}\big)
\leq C<\infty.
\end{equation*}
In the unweighted case(when $w$ equals a constant function), we denote the generalized unweighted weak Morrey space by $W\mathcal M^{1,\theta}(\mathbb R^n)$. That is, let $p=1$ and $\theta$ satisfy the $\mathcal D_\kappa$ condition $(\ref{D condition})$ with $0\leq\kappa<1$, we define
\begin{equation*}
W\mathcal M^{1,\theta}(\mathbb R^n):=\left\{f:
\big\|f\big\|_{W\mathcal M^{1,\theta}}=\sup_B\sup_{\sigma>0}\frac{1}{\theta(|B|)}\sigma\cdot\big|\big\{x\in B:|f(x)|>\sigma\big\}\big|<\infty\right\}.
\end{equation*}
\end{defn}
Note that
\begin{itemize}
  \item If $\theta(x)\equiv 1$, then $\mathcal M^{p,\theta}(w)=L^p_w(\mathbb R^n)$ and $W\mathcal M^{p,\theta}(w)=WL^p_w(\mathbb R^n)$. Thus our (weak) Morrey type space is an extension of the weighted (weak) Lebesgue space;
  \item If $\theta(x)=x^{\kappa}$ with $0<\kappa<1$, then $\mathcal M^{p,\theta}(w)$ is just the weighted Morrey space $L^{p,\kappa}(w)$, and $W\mathcal M^{1,\theta}(w)$ is just the weighted weak Morrey space $WL^{1,\kappa}(w)$;
  \item If $w$ equals a constant function, below we will show that $\mathcal M^{p,\theta}(\mathbb R^n)$ reduces to the generalized Morrey space $L^{p,\Theta}(\mathbb R^n)$, and $W\mathcal M^{1,\theta}(\mathbb R^n)$ reduces to the generalized weak Morrey space $WL^{1,\Theta}(\mathbb R^n)$.
\end{itemize}

\subsection{Main theorems}

\newtheorem{theorem}{Theorem}[section]

\newtheorem{corollary}[theorem]{Corollary}

The main results of this paper can be stated as follows.
\begin{theorem}\label{mainthm:1}
Let $0<\alpha\le1$, $1<p<\infty$ and $w\in A_p$. Assume that $\theta$ satisfies the $\mathcal D_\kappa$ condition $(\ref{D condition})$ with $0\leq\kappa<1$, then there is a constant $C>0$ independent of $\vec{f}=(f_1,f_2,\ldots)$ such that
\begin{equation*}
\bigg\|\bigg(\sum_{j}\big|\mathcal S_\alpha(f_j)\big|^2\bigg)^{1/2}\bigg\|_{\mathcal M^{p,\theta}(w)}
\leq C \bigg\|\bigg(\sum_{j}\big|f_j\big|^2\bigg)^{1/2}\bigg\|_{\mathcal M^{p,\theta}(w)}.
\end{equation*}
\end{theorem}

\begin{theorem}\label{mainthm:2}
Let $0<\alpha\le1$, $p=1$ and $w\in A_1$. Assume that $\theta$ satisfies the $\mathcal D_\kappa$ condition $(\ref{D condition})$ with $0\leq\kappa<1$, then there is a constant $C>0$ independent of $\vec{f}=(f_1,f_2,\ldots)$ such that
\begin{equation*}
\bigg\|\bigg(\sum_{j}\big|\mathcal S_\alpha(f_j)\big|^2\bigg)^{1/2}\bigg\|_{W\mathcal M^{1,\theta}(w)}
\leq C \bigg\|\bigg(\sum_{j}\big|f_j\big|^2\bigg)^{1/2}\bigg\|_{\mathcal M^{1,\theta}(w)}.
\end{equation*}
\end{theorem}

\begin{theorem}\label{mainthm:3}
Let $0<\alpha\le1$, $1<p<\infty$, $w\in A_p$ and $b\in BMO(\mathbb R^n)$. Assume that $\theta$ satisfies the $\mathcal D_\kappa$ condition $(\ref{D condition})$ with $0\leq\kappa<1$, then there is a constant $C>0$ independent of $\vec{f}=(f_1,f_2,\ldots)$ such that
\begin{equation*}
\bigg\|\bigg(\sum_{j}\big|\big[b,\mathcal S_\alpha\big](f_j)\big|^2\bigg)^{1/2}\bigg\|_{\mathcal M^{p,\theta}(w)}
\leq C \bigg\|\bigg(\sum_{j}\big|f_j\big|^2\bigg)^{1/2}\bigg\|_{\mathcal M^{p,\theta}(w)}.
\end{equation*}
\end{theorem}

In order to simplify the notations, for any given $\sigma>0$, we set
\begin{equation*}
\Phi\left(\frac{|f(x)|}{\sigma}\right)=\frac{|f(x)|}{\sigma}\cdot\left(1+\log^+\frac{|f(x)|}{\sigma}\right)
\end{equation*}
when $\Phi(t)=t\cdot(1+\log^+t)$ and $\log^+t=\max\{\log t,0\}$. For the endpoint estimates for these commutator operators in the weighted Lebesgue space $L^1_w(\mathbb R^n)$, we will show

\begin{theorem}\label{mainthm:4}
Let $0<\alpha\leq1$, $p=1$, $w\in A_1$ and $b\in BMO(\mathbb R^n)$. Then for any given $\sigma>0$, there exists a constant $C>0$ independent of $\vec{f}=(f_1,f_2,\ldots)$ and $\sigma>0$ such that
\begin{equation*}
\begin{split}
w\bigg(\bigg\{x\in\mathbb R^n:\bigg(\sum_{j}\big|\big[b,\mathcal S_{\alpha}\big](f_j)(x)\big|^2\bigg)^{1/2}>\sigma\bigg\}\bigg)
\leq C\int_{\mathbb R^n}\Phi\bigg(\frac{\|\vec{f}(x)\|_{\ell^2}}{\sigma}\bigg)\cdot w(x)\,dx,
\end{split}
\end{equation*}
where $\Phi(t)=t\cdot(1+\log^+t)$ and $\big\|\vec{f}(x)\big\|_{\ell^2}=\Big(\sum_{j}|f_j(x)|^2\Big)^{1/2}$.
\end{theorem}
For the endpoint estimates of commutators generated by $BMO(\mathbb R^n)$ functions and vector-valued intrinsic square functions in the Morrey type spaces associated to $\theta$, we will prove

\begin{theorem}\label{mainthm:5}
Let $0<\alpha\leq1$, $p=1$, $w\in A_1$ and $b\in BMO(\mathbb R^n)$. Assume that $\theta$ satisfies the $\mathcal D_\kappa$ condition $(\ref{D condition})$ with $0\leq\kappa<1$, then for any given $\sigma>0$ and any ball $B$ in $\mathbb R^n$, there exists a constant $C>0$ independent of $\vec{f}=(f_1,f_2,\ldots)$, $B$ and $\sigma>0$ such that
\begin{equation*}
\begin{split}
&\frac{1}{\theta(w(B))}\cdot w\bigg(\bigg\{x\in B:\bigg(\sum_{j}\big|\big[b,\mathcal S_{\alpha}\big](f_j)(x)\big|^2\bigg)^{1/2}>\sigma\bigg\}\bigg)\\
\leq &C\cdot\sup_B\left\{\frac{\Phi\Big(\frac{w(B)}{\theta(w(B))}\Big)}{w(B)}
\int_{B}\Phi\bigg(\frac{\|\vec{f}(x)\|_{\ell^2}}{\sigma}\bigg)\cdot w(x)\,dx\right\},
\end{split}
\end{equation*}
where $\Phi(t)=t\cdot(1+\log^+t)$ and $\big\|\vec{f}(x)\big\|_{\ell^2}=\Big(\sum_{j}|f_j(x)|^2\Big)^{1/2}$.
\end{theorem}

In particular, if we take $\theta(x)=x^\kappa$ with $0<\kappa<1$, then we immediately
get the following strong type estimates and endpoint estimates of vector-valued intrinsic square functions and commutators in the weighted Morrey spaces $L^{p,\kappa}(w)$ for all $0<\kappa<1$ and $1\leq p<\infty$.

\begin{corollary}
Let $0<\alpha\le1$, $1<p<\infty$, $0<\kappa<1$ and $w\in A_p$. Then there is a constant $C>0$ independent of $\vec{f}=(f_1,f_2,\ldots)$ such that
\begin{equation*}
\bigg\|\bigg(\sum_{j}\big|\mathcal S_\alpha(f_j)\big|^2\bigg)^{1/2}\bigg\|_{L^{p,\kappa}(w)}
\leq C\bigg\|\bigg(\sum_{j}\big|f_j\big|^2\bigg)^{1/2}\bigg\|_{L^{p,\kappa}(w)}.
\end{equation*}
\end{corollary}

\begin{corollary}
Let $0<\alpha\le1$, $p=1$, $0<\kappa<1$ and $w\in A_1$. Then there is a constant $C>0$ independent of $\vec{f}=(f_1,f_2,\ldots)$ such that
\begin{equation*}
\bigg\|\bigg(\sum_{j}\big|\mathcal S_\alpha(f_j)\big|^2\bigg)^{1/2}\bigg\|_{WL^{1,\kappa}(w)}
\leq C\bigg\|\bigg(\sum_{j}\big|f_j\big|^2\bigg)^{1/2}\bigg\|_{L^{1,\kappa}(w)}.
\end{equation*}
\end{corollary}

\begin{corollary}
Let $0<\alpha\le1$, $1<p<\infty$, $0<\kappa<1$, $w\in A_p$ and $b\in BMO(\mathbb R^n)$. Then there is a constant $C>0$ independent of $\vec{f}=(f_1,f_2,\ldots)$ such that
\begin{equation*}
\bigg\|\bigg(\sum_{j}\big|\big[b,\mathcal S_\alpha\big](f_j)\big|^2\bigg)^{1/2}\bigg\|_{L^{p,\kappa}(w)}
\leq C \bigg\|\bigg(\sum_{j}\big|f_j\big|^2\bigg)^{1/2}\bigg\|_{L^{p,\kappa}(w)}.
\end{equation*}
\end{corollary}

\begin{corollary}
Let $0<\alpha\leq1$, $p=1$, $0<\kappa<1$, $w\in A_1$ and $b\in BMO(\mathbb R^n)$. Then for any given $\sigma>0$ and any ball $B$ in $\mathbb R^n$, there exists a constant $C>0$ independent of $\vec{f}=(f_1,f_2,\ldots)$, $B$ and $\sigma>0$ such that
\begin{equation*}
\begin{split}
&\frac{1}{w(B)^\kappa}\cdot w\bigg(\bigg\{x\in B:\bigg(\sum_{j}\big|\big[b,\mathcal S_{\alpha}\big](f_j)(x)\big|^2\bigg)^{1/2}>\sigma\bigg\}\bigg)\\
\leq& C\cdot\sup_B\left\{\frac{\Phi\big(w(B)^{1-\kappa}\big)}{w(B)}
\int_{B}\Phi\bigg(\frac{\|\vec{f}(x)\|_{\ell^2}}{\sigma}\bigg)\cdot w(x)\,dx\right\},
\end{split}
\end{equation*}
where $\Phi(t)=t\cdot(1+\log^+t)$ and $\big\|\vec{f}(x)\big\|_{\ell^2}=\Big(\sum_{j}|f_j(x)|^2\Big)^{1/2}$.
\end{corollary}

We can also take $w$ to be a constant function, then we immediately get the following unweighted results.

\begin{corollary}
Let $0<\alpha\le1$ and $1<p<\infty$. Assume that $\theta$ satisfies the $\mathcal D_\kappa$ condition $(\ref{D condition})$ with $0\leq\kappa<1$, then there is a constant $C>0$ independent of $\vec{f}=(f_1,f_2,\ldots)$ such that
\begin{equation*}
\bigg\|\bigg(\sum_{j}\big|\mathcal S_\alpha(f_j)\big|^2\bigg)^{1/2}\bigg\|_{\mathcal M^{p,\theta}}
\leq C \bigg\|\bigg(\sum_{j}\big|f_j\big|^2\bigg)^{1/2}\bigg\|_{\mathcal M^{p,\theta}}.
\end{equation*}
\end{corollary}

\begin{corollary}
Let $0<\alpha\le1$ and $p=1$. Assume that $\theta$ satisfies the $\mathcal D_\kappa$ condition $(\ref{D condition})$ with $0\leq\kappa<1$, then there is a constant $C>0$ independent of $\vec{f}=(f_1,f_2,\ldots)$ such that
\begin{equation*}
\bigg\|\bigg(\sum_{j}\big|\mathcal S_\alpha(f_j)\big|^2\bigg)^{1/2}\bigg\|_{W\mathcal M^{1,\theta}}
\leq C \bigg\|\bigg(\sum_{j}\big|f_j\big|^2\bigg)^{1/2}\bigg\|_{\mathcal M^{1,\theta}}.
\end{equation*}
\end{corollary}

\begin{corollary}
Let $0<\alpha\le1$, $1<p<\infty$ and $b\in BMO(\mathbb R^n)$. Assume that $\theta$ satisfies the $\mathcal D_\kappa$ condition $(\ref{D condition})$ with $0\leq\kappa<1$, then there is a constant $C>0$ independent of $\vec{f}=(f_1,f_2,\ldots)$ such that
\begin{equation*}
\bigg\|\bigg(\sum_{j}\big|\big[b,\mathcal S_\alpha\big](f_j)\big|^2\bigg)^{1/2}\bigg\|_{\mathcal M^{p,\theta}}
\leq C \bigg\|\bigg(\sum_{j}\big|f_j\big|^2\bigg)^{1/2}\bigg\|_{\mathcal M^{p,\theta}}.
\end{equation*}
\end{corollary}

\begin{corollary}
Let $0<\alpha\leq1$, $p=1$ and $b\in BMO(\mathbb R^n)$. Assume that $\theta$ satisfies the $\mathcal D_\kappa$ condition $(\ref{D condition})$ with $0\leq\kappa<1$, then for any given $\sigma>0$ and any ball $B$ in $\mathbb R^n$, there exists a constant $C>0$ independent of $\vec{f}=(f_1,f_2,\ldots)$, $B$ and $\sigma>0$ such that
\begin{equation*}
\begin{split}
&\frac{1}{\theta(|B|)}\cdot\bigg|\bigg\{x\in B:\bigg(\sum_{j}\big|\big[b,\mathcal S_{\alpha}\big](f_j)(x)\big|^2\bigg)^{1/2}>\sigma\bigg\}\bigg|\\
\leq &C\cdot\sup_B\left\{\frac{\Phi\Big(\frac{|B|}{\theta(|B|)}\Big)}{|B|}
\int_{B}\Phi\bigg(\frac{\|\vec{f}(x)\|_{\ell^2}}{\sigma}\bigg)dx\right\},
\end{split}
\end{equation*}
where $\Phi(t)=t\cdot(1+\log^+t)$ and $\big\|\vec{f}(x)\big\|_{\ell^2}=\Big(\sum_{j}|f_j(x)|^2\Big)^{1/2}$.
\end{corollary}

Let $\Theta=\Theta(r)$, $r>0$, be a growth function with doubling constant $D(\Theta):1\le D(\Theta)<2^n$. If for any fixed $x_0\in\mathbb R^n$, we set $\theta(|B(x_0,r)|)=\Theta(r)$, then
\begin{equation*}
\theta(2^n|B(x_0,r)|)=\theta(|B(x_0,2r)|)=\Theta(2r).
\end{equation*}
For the doubling constant $D(\Theta)$ satisfying $1\le D(\Theta)<2^n$, which means that $D(\Theta)=2^{\kappa\cdot n}$ for some $0\leq\kappa<1$, then we are able to verify that $\theta$ is an increasing function and satisfies the $\mathcal D_\kappa$ condition $(\ref{D condition})$ with some $0\leq\kappa<1$. Thus, by the above unweighted results(Corollaries 2.10 through 2.13), we can also obtain strong type estimates and endpoint estimates of vector-valued intrinsic square functions and commutators in the generalized Morrey spaces $L^{p,\Theta}(\mathbb R^n)$ when $1\leq p<\infty$ and $\Theta$ satisfies the doubling condition $(\ref{doubling})$.

\begin{corollary}
Let $0<\alpha\le1$ and $1<p<\infty$. Suppose that $\Theta$ satisfies the doubling condition $(\ref{doubling})$ and $1\le D(\Theta)<2^n$, then there is a constant $C>0$ independent of $\vec{f}=(f_1,f_2,\ldots)$ such that
\begin{equation*}
\bigg\|\bigg(\sum_{j}\big|\mathcal S_\alpha(f_j)\big|^2\bigg)^{1/2}\bigg\|_{L^{p,\Theta}}
\leq C\bigg\|\bigg(\sum_{j}\big|f_j\big|^2\bigg)^{1/2}\bigg\|_{L^{p,\Theta}}.
\end{equation*}
\end{corollary}

\begin{corollary}
Let $0<\alpha\le1$ and $p=1$. Suppose that $\Theta$ satisfies the doubling condition $(\ref{doubling})$ and $1\le D(\Theta)<2^n$, then there is a constant $C>0$ independent of $\vec{f}=(f_1,f_2,\ldots)$ such that
\begin{equation*}
\bigg\|\bigg(\sum_{j}\big|\mathcal S_\alpha(f_j)\big|^2\bigg)^{1/2}\bigg\|_{WL^{1,\Theta}}
\leq C\bigg\|\bigg(\sum_{j}\big|f_j\big|^2\bigg)^{1/2}\bigg\|_{L^{1,\Theta}}.
\end{equation*}
\end{corollary}

\begin{corollary}
Let $0<\alpha\le1$, $1<p<\infty$ and $b\in BMO(\mathbb R^n)$. Suppose that $\Theta$ satisfies the doubling condition $(\ref{doubling})$ and $1\le D(\Theta)<2^n$, then there is a constant $C>0$ independent of $\vec{f}=(f_1,f_2,\ldots)$ such that
\begin{equation*}
\bigg\|\bigg(\sum_{j}\big|\big[b,\mathcal S_\alpha\big](f_j)\big|^2\bigg)^{1/2}\bigg\|_{L^{p,\Theta}}
\leq C\bigg\|\bigg(\sum_{j}\big|f_j\big|^2\bigg)^{1/2}\bigg\|_{L^{p,\Theta}}.
\end{equation*}
\end{corollary}

\begin{corollary}
Let $0<\alpha\leq1$, $p=1$ and $b\in BMO(\mathbb R^n)$. Suppose that $\Theta$ satisfies the doubling condition $(\ref{doubling})$ and $1\le D(\Theta)<2^n$, then for any given $\sigma>0$ and any ball $B(x_0,r)$ in $\mathbb R^n$, there exists a constant $C>0$ independent of $\vec{f}=(f_1,f_2,\ldots)$, $B(x_0,r)$ and $\sigma>0$ such that
\begin{equation*}
\begin{split}
&\frac{1}{\Theta(r)}\cdot\bigg|\bigg\{x\in B(x_0,r):\bigg(\sum_{j}\big|\big[b,\mathcal S_{\alpha}\big](f_j)(x)\big|^2\bigg)^{1/2}>\sigma\bigg\}\bigg|\\
\leq& C\cdot\sup_{r>0}\left\{\frac{\Phi\Big(\frac{|B(x_0,r)|}{\Theta(r)}\Big)}{|B(x_0,r)|}
\int_{B(x_0,r)}\Phi\bigg(\frac{\|\vec{f}(x)\|_{\ell^2}}{\sigma}\bigg)dx\right\},
\end{split}
\end{equation*}
where $\Phi(t)=t\cdot(1+\log^+t)$ and $\big\|\vec{f}(x)\big\|_{\ell^2}=\Big(\sum_{j}|f_j(x)|^2\Big)^{1/2}$.
\end{corollary}

Throughout this paper, the letter $C$ always denotes a positive constant independent of the main parameters involved, but it may be different from line to line.

\section{Proofs of Theorems \ref{mainthm:1} and \ref{mainthm:2}}

\begin{proof}[Proof of Theorem $\ref{mainthm:1}$]

Let $\big(\sum_{j}|f_j|^2\big)^{1/2}\in\mathcal M^{p,\theta}(w)$ with $1<p<\infty$ and $w\in A_p$. For arbitrary $x_0\in\mathbb R^n$, set $B=B(x_0,r_B)$ for the ball centered at $x_0$ and of radius $r_B$. We represent $f_j$ as
\begin{equation*}
f_j=f_j\cdot\chi_{2B}+f_j\cdot\chi_{(2B)^c}:=f^0_j+f^\infty_j,
\end{equation*}
where $\chi_{2B}$ denotes the characteristic function of $2B=B(x_0,2r_B)\subseteq\mathbb R^n$, $j=1,2,\ldots$. Then we write
\begin{equation*}
\begin{split}
&\frac{1}{\theta(w(B))^{1/p}}\Bigg(\int_B\bigg(\sum_{j}\big|\mathcal S_\alpha(f_j)(x)\big|^2\bigg)^{p/2}w(x)\,dx\Bigg)^{1/p}\\
\leq\ &\frac{1}{\theta(w(B))^{1/p}}\Bigg(\int_B\bigg(\sum_{j}\big|\mathcal S_\alpha(f^0_j)(x)\big|^2\bigg)^{p/2}w(x)\,dx\Bigg)^{1/p}\\
&+\frac{1}{\theta(w(B))^{1/p}}\Bigg(\int_B\bigg(\sum_{j}\big|\mathcal S_\alpha(f^\infty_j)(x)\big|^2\bigg)^{p/2}w(x)\,dx\Bigg)^{1/p}\\
:=\ &I_1+I_2.
\end{split}
\end{equation*}
Let us first estimate $I_1$. From the boundedness of vector-valued intrinsic square functions in $L^p_w(\mathbb R^n)$ (see Theorem A), it follows that
\begin{equation*}
\begin{split}
I_1&\leq \frac{1}{\theta(w(B))^{1/p}}
\bigg\|\bigg(\sum_{j}\big|\mathcal S_\alpha(f^0_j)\big|^2\bigg)^{1/2}\bigg\|_{L^p_w}\\
&\leq C\cdot\frac{1}{\theta(w(B))^{1/p}}
\Bigg(\int_{2B}\bigg(\sum_{j}\big|f_j(x)\big|^2\bigg)^{p/2}w(x)\,dx\Bigg)^{1/p}\\
&\leq C\bigg\|\bigg(\sum_{j}\big|f_j\big|^2\bigg)^{1/2}\bigg\|_{\mathcal M^{p,\theta}(w)}
\cdot\frac{\theta(w(2B))^{1/p}}{\theta(w(B))^{1/p}}.
\end{split}
\end{equation*}
Moreover, since $0<w(B)<w(2B)<+\infty$ when $w\in A_p$ with $1<p<\infty$, then by the $\mathcal D_\kappa$ condition (\ref{D condition}) of $\theta$ and the inequality \eqref{weights}, we obtain
\begin{equation*}
\begin{split}
I_1&\leq  C\bigg\|\bigg(\sum_{j}\big|f_j\big|^2\bigg)^{1/2}\bigg\|_{\mathcal M^{p,\theta}(w)}\cdot\frac{w(2B)^{\kappa/p}}{w(B)^{\kappa/p}}\\
&\leq C\bigg\|\bigg(\sum_{j}\big|f_j\big|^2\bigg)^{1/2}\bigg\|_{\mathcal M^{p,\theta}(w)}.
\end{split}
\end{equation*}
We now consider the other term $I_2$. For any $\varphi\in{\mathcal C}_\alpha$, $0<\alpha\le1$, $j=1,2,\ldots$, and $(y,t)\in\Gamma(x)$ with $x\in B$, we have
\begin{align}\label{Key1}
\bigg|\int_{\mathbb R^n}\varphi_t(y-z)f^\infty_j(z)\,dz\bigg|
&=\bigg|\int_{(2B)^c}\varphi_t(y-z)f_j(z)\,dz\bigg|\notag\\
&\leq C\cdot t^{-n}\int_{(2B)^c\cap\{z:|y-z|\le t\}}\big|f_j(z)\big|\,dz\notag\\
&\leq C\cdot t^{-n}\sum_{l=1}^\infty\int_{(2^{l+1}B\backslash 2^{l}B)\cap\{z:|y-z|\le t\}}\big|f_j(z)\big|\,dz.
\end{align}
Since $|y-z|\le t$ and $(y,t)\in\Gamma(x)$, then one has $|x-z|\leq|x-y|+|y-z|\leq 2t$. Hence, for any $x\in B$ and $z\in\big(2^{l+1}B\backslash 2^{l}B\big)$, a direct computation shows that
\begin{equation}\label{Key2}
2t\geq |x-z|\geq |z-x_0|-|x-x_0|\geq 2^{l-1}r_B.
\end{equation}
Therefore, by using the above inequalities \eqref{Key1} and \eqref{Key2} together with Minkowski's inequality for integrals, we can deduce
\begin{equation*}
\begin{split}
\mathcal S_\alpha(f^\infty_j)(x)&=\left(\iint_{\Gamma(x)}\sup_{\varphi\in{\mathcal C}_\alpha}\bigg|\int_{\mathbb R^n} \varphi_t(y-z)f^\infty_j(z)\,dz\bigg|^2\frac{dydt}{t^{n+1}}\right)^{1/2}\\
&\leq C\left(\int_{2^{l-2}r_B}^\infty\int_{|x-y|<t}\bigg|t^{-n}\sum_{l=1}^\infty\int_{2^{l+1}B\backslash 2^{l}B} \big|f_j(z)\big|\,dz\bigg|^2\frac{dydt}{t^{n+1}}\right)^{1/2}\\
&\le C\left(\sum_{l=1}^\infty\int_{2^{l+1}B\backslash 2^{l}B}\big|f_j(z)\big|\,dz\right)
\left(\int_{2^{l-2}r_B}^\infty\frac{dt}{t^{2n+1}}\right)^{1/2}\\
&\leq C\sum_{l=1}^\infty\frac{1}{|2^{l+1}B|}\int_{2^{l+1}B\backslash 2^{l}B}\big|f_j(z)\big|\,dz.
\end{split}
\end{equation*}
Then by duality and Cauchy--Schwarz inequality, we get
\begin{align}\label{key estimate1}
&\bigg(\sum_{j}\big|\mathcal S_\alpha(f^\infty_j)(x)\big|^2\bigg)^{1/2}\notag\\
&\leq C\Bigg(\sum_{j}\bigg|\sum_{l=1}^\infty\frac{1}{|2^{l+1}B|}
\int_{2^{l+1}B\backslash 2^{l}B}\big|f_j(z)\big|\,dz\bigg|^2\Bigg)^{1/2}\notag\\
&\leq C\sup_{(\sum_j|\zeta_j|^2)^{1/2}\leq1}\sum_{j}\bigg(\sum_{l=1}^\infty\frac{1}{|2^{l+1}B|}
\int_{2^{l+1}B}\big|f_j(z)\big|\,dz\cdot\zeta_j\bigg)\notag\\
&\leq C\sum_{l=1}^\infty\frac{1}{|2^{l+1}B|}\int_{2^{l+1}B}\sup_{(\sum_j|\zeta_j|^2)^{1/2}\leq1}
\bigg(\sum_{j}\big|f_j(z)\big|\cdot\zeta_j\bigg)dz\notag\\
&\leq C\sum_{l=1}^\infty\frac{1}{|2^{l+1}B|}\int_{2^{l+1}B}\bigg(\sum_{j}\big|f_j(z)\big|^2\bigg)^{1/2}dz.
\end{align}
Furthermore, it follows from H\"older's inequality and $A_p$ condition on $w$ that
\begin{equation*}
\begin{split}
&\bigg(\sum_{j}\big|\mathcal S_\alpha(f^\infty_j)(x)\big|^2\bigg)^{1/2}\\
&\leq C\sum_{l=1}^\infty\frac{1}{|2^{l+1}B|}\Bigg(\int_{2^{l+1}B}
\bigg(\sum_{j}\big|f_j(z)\big|^2\bigg)^{p/2}w(z)\,dz\Bigg)^{1/p}
\times\left(\int_{2^{l+1}B}w(z)^{-{p'}/p}\,dz\right)^{1/{p'}}\\
&\leq C\bigg\|\bigg(\sum_{j}\big|f_j\big|^2\bigg)^{1/2}\bigg\|_{\mathcal M^{p,\theta}(w)}
\times\sum_{l=1}^\infty\frac{\theta(w(2^{l+1}B))^{1/p}}{w(2^{l+1}B)^{1/p}}.
\end{split}
\end{equation*}
Hence, by the above pointwise estimate,
\begin{equation*}
\begin{split}
I_2&\leq C\bigg\|\bigg(\sum_{j}\big|f_j\big|^2\bigg)^{1/2}\bigg\|_{\mathcal M^{p,\theta}(w)}
\times\sum_{l=1}^\infty\frac{\theta(w(2^{l+1}B))^{1/p}}{\theta(w(B))^{1/p}}\cdot\frac{w(B)^{1/p}}{w(2^{l+1}B)^{1/p}}.
\end{split}
\end{equation*}
For any $l\in\mathbb Z^+$, since $0<w(B)<w(2^{l+1}B)<+\infty$ when $w\in A_p\subset A_\infty$ with $1<p<\infty$, then by using the $\mathcal D_\kappa$ condition (\ref{D condition}) of $\theta$ again, the inequality (\ref{compare}) with exponent $\delta>0$ and the fact that $0\leq\kappa<1$, we find that
\begin{align}\label{theta1}
\sum_{l=1}^\infty\frac{\theta(w(2^{l+1}B))^{1/p}}{\theta(w(B))^{1/p}}\cdot\frac{w(B)^{1/p}}{w(2^{l+1}B)^{1/p}}
&\leq C\sum_{l=1}^\infty\frac{w(B)^{{(1-\kappa)}/p}}{w(2^{l+1}B)^{{(1-\kappa)}/p}}\notag\\
&\leq C\sum_{l=1}^\infty\left(\frac{|B|}{|2^{l+1}B|}\right)^{\delta {(1-\kappa)}/p}\notag\\
&\leq C\sum_{l=1}^\infty\left(\frac{1}{2^{(l+1)n}}\right)^{\delta {(1-\kappa)}/p}\notag\\
&\leq C,
\end{align}
where the last series is convergent since the exponent $\delta {(1-\kappa)}/p$ is positive. This implies our desired estimate
\begin{equation*}
I_2\leq C\bigg\|\bigg(\sum_{j}\big|f_j\big|^2\bigg)^{1/2}\bigg\|_{\mathcal M^{p,\theta}(w)}.
\end{equation*}
Combining the above two estimates for $I_1$ and $I_2$, and then taking the supremum over all balls $B\subset\mathbb R^n$, we complete the proof of Theorem \ref{mainthm:1}.
\end{proof}

\begin{proof}[Proof of Theorem $\ref{mainthm:2}$]
Let $\big(\sum_{j}|f_j|^2\big)^{1/2}\in\mathcal M^{1,\theta}(w)$ with $w\in A_1$. For arbitrary $x_0\in\mathbb R^n$, set $B=B(x_0,r_B)$ for the ball centered at $x_0$ and of radius $r_B$. Write $f_j=f^0_j+f^\infty_j$ with $f^0_j=f_j\cdot\chi_{2B}$ and $f^\infty_j=f_j\cdot\chi_{(2B)^c}$, $j=1,2,\ldots$. Then for any given $\sigma>0$, we have
\begin{equation*}
\begin{split}
&\frac{1}{\theta(w(B))}\sigma\cdot w\bigg(\bigg\{x\in B:\bigg(\sum_{j}\big|\mathcal S_{\alpha}(f_j)(x)\big|^2\bigg)^{1/2}>\sigma\bigg\}\bigg)\\
\end{split}
\end{equation*}
\begin{equation*}
\begin{split}
&\leq\frac{1}{\theta(w(B))}\sigma\cdot w\bigg(\bigg\{x\in B:\bigg(\sum_{j}\big|\mathcal S_{\alpha}(f^0_j)(x)\big|^2\bigg)^{1/2}>\sigma/2\bigg\}\bigg)\\
&+ \frac{1}{\theta(w(B))}\sigma\cdot w\bigg(\bigg\{x\in B:\bigg(\sum_{j}\big|\mathcal S_{\alpha}(f^{\infty}_j)(x)\big|^2\bigg)^{1/2}>\sigma/2\bigg\}\bigg)\\
:=&I'_1+I'_2.
\end{split}
\end{equation*}
Below we will give the estimates of $I'_1$ and $I'_2$, respectively. From the weighted weak $(1,1)$ boundedness of vector-valued intrinsic square functions (see Theorem B), it follows that
\begin{equation*}
\begin{split}
I'_1&\leq \frac{2}{\theta(w(B))}
\bigg\|\bigg(\sum_{j}\big|\mathcal S_\alpha(f^0_j)\big|^2\bigg)^{1/2}\bigg\|_{WL^1_w}\\
&\leq C\cdot\frac{1}{\theta(w(B))}\Bigg(\int_{2B}\bigg(\sum_{j}\big|f_j(x)\big|^2\bigg)^{1/2}w(x)\,dx\Bigg)\\
&\leq C\bigg\|\bigg(\sum_{j}\big|f_j\big|^2\bigg)^{1/2}\bigg\|_{\mathcal M^{1,\theta}(w)}
\cdot\frac{\theta(w(2B))}{\theta(w(B))}.
\end{split}
\end{equation*}
Moreover, since $0<w(B)<w(2B)<+\infty$ when $w\in A_1$, then by the $\mathcal D_\kappa$ condition (\ref{D condition}) of $\theta$ and inequality (\ref{weights}), we get
\begin{equation*}
\begin{split}
I'_1&\leq C\bigg\|\bigg(\sum_{j}\big|f_j\big|^2\bigg)^{1/2}\bigg\|_{\mathcal M^{1,\theta}(w)}
\cdot\frac{w(2B)^{\kappa}}{w(B)^{\kappa}}\\
&\leq C\bigg\|\bigg(\sum_{j}\big|f_j\big|^2\bigg)^{1/2}\bigg\|_{\mathcal M^{1,\theta}(w)}.
\end{split}
\end{equation*}
As for the term $I'_2$, it follows directly from Chebyshev's inequality and the pointwise estimate \eqref{key estimate1} that
\begin{align}\label{Iprime}
I'_2&\leq\frac{1}{\theta(w(B))}\sigma\cdot\frac{\,2\,}{\sigma}
\int_B\bigg(\sum_{j}\big|\mathcal S_{\alpha}(f^{\infty}_j)(x)\big|^2\bigg)^{1/2}w(x)\,dx\notag\\
&\leq C\cdot\frac{w(B)}{\theta(w(B))}
\sum_{l=1}^\infty\frac{1}{|2^{l+1}B|}\int_{2^{l+1}B}\bigg(\sum_{j}\big|f_j(z)\big|^2\bigg)^{1/2}dz.
\end{align}
Another application of $A_1$ condition on $w$ gives that
\begin{equation*}
\begin{split}
&\frac{1}{|2^{l+1}B|}\int_{2^{l+1}B}\bigg(\sum_{j}\big|f_j(z)\big|^2\bigg)^{1/2}dz\\
&\leq C\frac{1}{w(2^{l+1}B)}\cdot\underset{z\in 2^{l+1}B}{\mbox{ess\,inf}}\;w(z)
\int_{2^{l+1}B}\bigg(\sum_{j}\big|f_j(z)\big|^2\bigg)^{1/2}dz\\
&\leq C\frac{1}{w(2^{l+1}B)}\bigg(\int_{2^{l+1}B}\bigg(\sum_{j}\big|f_j(z)\big|^2\bigg)^{1/2}w(z)\,dz\bigg)\\
\end{split}
\end{equation*}
\begin{equation}\label{Iprime2}
\begin{split}
&\leq C\bigg\|\bigg(\sum_{j}\big|f_j\big|^2\bigg)^{1/2}\bigg\|_{\mathcal M^{1,\theta}(w)}\cdot\frac{\theta(w(2^{l+1}B))}{w(2^{l+1}B)}.
\end{split}
\end{equation}
Substituting the above inequality \eqref{Iprime2} into \eqref{Iprime}, we thus obtain
\begin{equation*}
\begin{split}
I'_2&\leq C\bigg\|\bigg(\sum_{j}\big|f_j\big|^2\bigg)^{1/2}\bigg\|_{\mathcal M^{1,\theta}(w)}
\times\sum_{l=1}^\infty\frac{\theta(w(2^{l+1}B))}{\theta(w(B))}\cdot\frac{w(B)}{w(2^{l+1}B)}.
\end{split}
\end{equation*}
Note that $w\in A_1\subset A_\infty$, then one has $0<w(B)<w(2^{l+1}B)<+\infty$ for any $l\in\mathbb Z^+$. Thus, by using the $\mathcal D_\kappa$ condition (\ref{D condition}) of $\theta$ again, the inequality \eqref{compare} with exponent $\delta^*>0$ and the fact that $0\leq\kappa<1$, we find that
\begin{align}\label{theta2}
\sum_{l=1}^\infty\frac{\theta(w(2^{l+1}B))}{\theta(w(B))}\cdot\frac{w(B)}{w(2^{l+1}B)}
&\leq C\sum_{l=1}^\infty\frac{w(B)^{1-\kappa}}{w(2^{l+1}B)^{1-\kappa}}\notag\\
&\leq C\sum_{l=1}^\infty\left(\frac{|B|}{|2^{l+1}B|}\right)^{\delta^*(1-\kappa)}\notag\\
&\leq C\sum_{l=1}^\infty\left(\frac{1}{2^{(l+1)n}}\right)^{\delta^*(1-\kappa)}\notag\\
&\leq C.
\end{align}
Therefore,
\begin{equation*}
I'_2\leq C\bigg\|\bigg(\sum_{j}\big|f_j\big|^2\bigg)^{1/2}\bigg\|_{\mathcal M^{1,\theta}(w)}.
\end{equation*}
Summing up the above estimates for $I'_1$ and $I'_2$, and then taking the supremum over all balls $B\subset\mathbb R^n$ and all $\sigma>0$, we finish the proof of Theorem \ref{mainthm:2}.
\end{proof}

\section{Proof of Theorem \ref{mainthm:3}}

Given a real-valued function $b\in BMO(\mathbb R^n)$, we will follow the idea developed in \cite{alvarez,ding} and denote $F(\xi)=e^{\xi[b(x)-b(z)]}$, $\xi\in\mathbb C$. Then by the analyticity of $F(\xi)$ on $\mathbb C$ and the Cauchy integral formula, we get
\begin{equation*}
\begin{split}
b(x)-b(z)&=F'(0)=\frac{1}{2\pi i}\int_{|\xi|=1}\frac{F(\xi)}{\xi^2}\,d\xi\\
&=\frac{1}{2\pi}\int_0^{2\pi}e^{e^{i\theta}[b(x)-b(z)]}\cdot e^{-i\theta}d\theta.
\end{split}
\end{equation*}
Thus, for any $\varphi\in{\mathcal C}_\alpha$, $0<\alpha\le1$ and $j\in\mathbb Z^+$, we obtain
\begin{equation*}
\begin{split}
&\bigg|\int_{\mathbb R^n}\big[b(x)-b(z)\big]\varphi_t(y-z)f_j(z)\,dz\bigg|\\
=&\bigg|\frac{1}{2\pi}\int_0^{2\pi}\bigg(\int_{\mathbb R^n}\varphi_t(y-z)e^{-e^{i\theta}b(z)}f_j(z)\,dz\bigg)
e^{e^{i\theta}b(x)}\cdot e^{-i\theta}d\theta\bigg|\\
\end{split}
\end{equation*}
\begin{equation*}
\begin{split}
\leq&\frac{1}{2\pi}\int_0^{2\pi}\sup_{\varphi\in{\mathcal C}_\alpha}\bigg|\int_{\mathbb R^n}\varphi_t(y-z)e^{-e^{i\theta}b(z)}f_j(z)\,dz\bigg|
e^{\cos\theta\cdot b(x)}d\theta\\
\leq&\frac{1}{2\pi}\int_0^{2\pi}A_\alpha\big(e^{-e^{i\theta}b}\cdot f_j\big)(y,t)\cdot e^{\cos\theta\cdot b(x)}d\theta.
\end{split}
\end{equation*}
So we have
\begin{equation*}
\big|\big[b,\mathcal S_\alpha\big](f_j)(x)\big|\leq\frac{1}{2\pi}\int_0^{2\pi}
\mathcal S_\alpha\big(e^{-e^{i\theta}b}\cdot f_j\big)(x)\cdot e^{\cos\theta\cdot b(x)}d\theta.
\end{equation*}
Moreover, by using standard duality argument and Cauchy--Schwarz inequality, we get
\begin{equation*}
\begin{split}
&\bigg(\sum_{j}\big|\big[b,\mathcal S_\alpha\big](f_j)(x)\big|^2\bigg)^{1/2}\\
&\leq\frac{1}{2\pi}\Bigg(\sum_{j}\left|\int_0^{2\pi}
\mathcal S_\alpha\big(e^{-e^{i\theta}b}\cdot f_j\big)(x)\cdot e^{\cos\theta\cdot b(x)}d\theta\right|^2\Bigg)^{1/2}\\
&\leq\frac{1}{2\pi}\sup_{(\sum_j|\zeta_j|^2)^{1/2}\leq1}\sum_{j}\left(\int_0^{2\pi}
\mathcal S_\alpha\big(e^{-e^{i\theta}b}\cdot f_j\big)(x)\cdot e^{\cos\theta\cdot b(x)}d\theta\cdot\zeta_j\right)\\
&\leq\frac{1}{2\pi}\int_0^{2\pi}\sup_{(\sum_j|\zeta_j|^2)^{1/2}\leq1}
\Bigg(\sum_{j}\mathcal S_\alpha\big(e^{-e^{i\theta}b}\cdot f_j\big)(x)\cdot e^{\cos\theta\cdot b(x)}\cdot\zeta_j\Bigg)d\theta\\
&\leq\frac{1}{2\pi}\int_0^{2\pi}\bigg(\sum_{j}\Big|\mathcal S_\alpha\big(e^{-e^{i\theta}b}\cdot f_j\big)(x)\Big|^2\bigg)^{1/2}
\cdot e^{\cos\theta\cdot b(x)}d\theta.
\end{split}
\end{equation*}
Therefore, by the $L^p_w$-boundedness of vector-valued intrinsic square functions (see Theorem A), and using the same arguments as in \cite{ding}, we can also show the following result.
\begin{theorem}\label{commutator thm}
Let $0<\alpha\le1$, $1<p<\infty$ and $w\in A_p$. Then there exists a constant $C>0$ independent of $\vec{f}=(f_1,f_2,\ldots)$ such that
\begin{equation*}
\bigg\|\bigg(\sum_{j}\big|\big[b,\mathcal S_\alpha\big](f_j)\big|^2\bigg)^{1/2}\bigg\|_{L^p_w}
\leq C \bigg\|\bigg(\sum_{j}\big|f_j\big|^2\bigg)^{1/2}\bigg\|_{L^p_w}
\end{equation*}
provided that $b\in BMO(\mathbb R^n)$.
\end{theorem}

We are now in a position to give the proof of Theorem $\ref{mainthm:3}$.

\begin{proof}
Let $\big(\sum_{j}|f_j|^2\big)^{1/2}\in\mathcal M^{p,\theta}(w)$ with $1<p<\infty$ and $w\in A_p$. Fix $x_0\in\mathbb R^n$ and let $B=B(x_0,r_B)$ be a ball centered at $x_0$ of radius $r_B$. We split $f_j$ by $f_j=f^0_j+f^\infty_j$, where $f^0_j=f_j\cdot\chi_{2B}$ and $2B=B(x_0,2r_B)\subseteq\mathbb R^n$, $j=1,2,\ldots$. Then we write
\begin{equation*}
\begin{split}
&\frac{1}{\theta(w(B))^{1/p}}\Bigg(\int_B\bigg(\sum_{j}\big|\big[b,\mathcal S_\alpha\big](f_j)(x)\big|^2\bigg)^{p/2}w(x)\,dx\Bigg)^{1/p}\\
\end{split}
\end{equation*}
\begin{equation*}
\begin{split}
\leq\ &\frac{1}{\theta(w(B))^{1/p}}\Bigg(\int_B\bigg(\sum_{j}\big|\big[b,\mathcal S_\alpha\big](f^0_j)(x)\big|^2\bigg)^{p/2}w(x)\,dx\Bigg)^{1/p}\\
&+\frac{1}{\theta(w(B))^{1/p}}\Bigg(\int_B\bigg(\sum_{j}\big|\big[b,\mathcal S_\alpha\big](f^\infty_j)(x)\big|^2\bigg)^{p/2}w(x)\,dx\Bigg)^{1/p}\\
:=\ &J_1+J_2.
\end{split}
\end{equation*}
By using Theorem \ref{commutator thm}, the $\mathcal D_\kappa$ condition (\ref{D condition}) of $\theta$ and the inequality (\ref{weights}), we obtain
\begin{equation*}
\begin{split}
J_1&\leq \frac{1}{\theta(w(B))^{1/p}}
\bigg\|\bigg(\sum_{j}\big|\big[b,\mathcal S_\alpha\big](f^0_j)\big|^2\bigg)^{1/2}\bigg\|_{L^p_w}\\
&\leq C\cdot\frac{1}{\theta(w(B))^{1/p}}
\Bigg(\int_{2B}\bigg(\sum_{j}\big|f_j(x)\big|^2\bigg)^{p/2}w(x)\,dx\Bigg)^{1/p}\\
&\leq C\bigg\|\bigg(\sum_{j}\big|f_j\big|^2\bigg)^{1/2}\bigg\|_{\mathcal M^{p,\theta}(w)}
\cdot\frac{\theta(w(2B))^{1/p}}{\theta(w(B))^{1/p}}\\
&\leq C\bigg\|\bigg(\sum_{j}\big|f_j\big|^2\bigg)^{1/2}\bigg\|_{\mathcal M^{p,\theta}(w)}\cdot\frac{w(2B)^{\kappa/p}}{w(B)^{\kappa/p}}\\
&\leq C\bigg\|\bigg(\sum_{j}\big|f_j\big|^2\bigg)^{1/2}\bigg\|_{\mathcal M^{p,\theta}(w)}.
\end{split}
\end{equation*}
Let us now turn to estimate the other term $J_2$. For any given $x\in B$, $(y,t)\in\Gamma(x)$ and for $j=1,2,\dots$, we have
\begin{equation*}
\begin{split}
\sup_{\varphi\in{\mathcal C}_\alpha}\bigg|\int_{\mathbb R^n}\big[b(x)-b(z)\big]\varphi_t(y-z)f^\infty_j(z)\,dz\bigg|&\le
\big|b(x)-b_B\big|\cdot\sup_{\varphi\in{\mathcal C}_\alpha}\bigg|\int_{\mathbb R^n}\varphi_t(y-z)f^\infty_j(z)\,dz\bigg|\\
&+\sup_{\varphi\in{\mathcal C}_\alpha}\bigg|\int_{\mathbb R^n}\big[b_B-b(z)\big]\varphi_t(y-z)f^\infty_j(z)\,dz\bigg|.
\end{split}
\end{equation*}
By definition, we thus have
\begin{equation*}
\begin{split}
\big|\big[b,\mathcal S_\alpha\big](f^\infty_j)(x)\big|&\leq\big|b(x)-b_B\big|\cdot\mathcal S_\alpha(f^\infty_j)(x)
+\mathcal S_\alpha\Big([b_{B}-b]f^\infty_j\Big)(x).
\end{split}
\end{equation*}
From this and Minkowski' inequality for series, we further obtain
\begin{equation*}
\begin{split}
\bigg(\sum_{j}\big|\big[b,\mathcal S_\alpha\big](f^\infty_j)(x)\big|^2\bigg)^{1/2}
&\leq\big|b(x)-b_B\big|\bigg(\sum_{j}\big|\mathcal S_\alpha(f^{\infty}_j)(x)\big|^2\bigg)^{1/2}\\
&+\bigg(\sum_{j}\Big|\mathcal S_\alpha\Big([b_{B}-b]f^\infty_j\Big)(x)\Big|^2\bigg)^{1/2}.
\end{split}
\end{equation*}
For any $\varphi\in{\mathcal C}_\alpha$, $0<\alpha\le1$, $j=1,2,\ldots$, and $(y,t)\in\Gamma(x)$ with $x\in B$, we have
\begin{align}\label{Key3}
&\bigg|\int_{\mathbb R^n}\big[b_B-b(z)\big]\varphi_t(y-z)f^\infty_j(z)\,dz\bigg|\notag\\
&=\bigg|\int_{(2B)^c}\big[b_B-b(z)\big]\varphi_t(y-z)f_j(z)\,dz\bigg|\notag\\
&\leq C\cdot t^{-n}\int_{(2B)^c\cap\{z:|y-z|\le t\}}\big|b(z)-b_B\big|\big|f_j(z)\big|\,dz\notag\\
&\leq C\cdot t^{-n}\sum_{l=1}^\infty\int_{(2^{l+1}B\backslash 2^{l}B)\cap\{z:|y-z|\le t\}}\big|b(z)-b_B\big|\big|f_j(z)\big|\,dz.
\end{align}
Hence, for any $x\in B$, by using the inequalities \eqref{Key3} and \eqref{Key2} together with Minkowski's inequality for integrals, we can deduce that
\begin{equation*}
\begin{split}
&\mathcal S_\alpha\Big([b_{B}-b]f^\infty_j\Big)(x)\\
&=\left(\iint_{\Gamma(x)}\sup_{\varphi\in{\mathcal C}_\alpha}\bigg|\int_{\mathbb R^n}\big[b_B-b(z)\big]\varphi_t(y-z)f^\infty_j(z)\,dz\bigg|^2\frac{dydt}{t^{n+1}}\right)^{1/2}\\
&\leq C\left(\int_{2^{l-2}r_B}^\infty\int_{|x-y|<t}\bigg|t^{-n}\sum_{l=1}^\infty\int_{2^{l+1}B\backslash 2^{l}B}\big|b(z)-b_B\big|\big|f_j(z)\big|\,dz\bigg|^2\frac{dydt}{t^{n+1}}\right)^{1/2}\\
&\le C\left(\sum_{l=1}^\infty\int_{2^{l+1}B\backslash 2^{l}B}\big|b(z)-b_B\big|\big|f_j(z)\big|\,dz\right)
\left(\int_{2^{l-2}r_B}^\infty\frac{dt}{t^{2n+1}}\right)^{1/2}\\
&\leq C\sum_{l=1}^\infty\frac{1}{|2^{l+1}B|}\int_{2^{l+1}B\backslash 2^{l}B}\big|b(z)-b_B\big|\big|f_j(z)\big|\,dz.
\end{split}
\end{equation*}
Therefore, by duality and Cauchy--Schwarz inequality, we get
\begin{align}\label{key estimate2}
&\bigg(\sum_{j}\Big|\mathcal S_\alpha\Big([b_{B}-b]f^\infty_j\Big)(x)\Big|^2\bigg)^{1/2}\notag\\
&\leq C\Bigg(\sum_{j}\bigg|\sum_{l=1}^\infty\frac{1}{|2^{l+1}B|}
\int_{2^{l+1}B\backslash 2^{l}B}\big|b(z)-b_B\big|\big|f_j(z)\big|\,dz\bigg|^2\Bigg)^{1/2}\notag\\
&\leq C\sup_{(\sum_j|\zeta_j|^2)^{1/2}\leq1}\sum_{j}\bigg(\sum_{l=1}^\infty\frac{1}{|2^{l+1}B|}
\int_{2^{l+1}B}\big|b(z)-b_B\big|\big|f_j(z)\big|\,dz\cdot\zeta_j\bigg)\notag\\
&\leq C\sum_{l=1}^\infty\frac{1}{|2^{l+1}B|}\int_{2^{l+1}B}\sup_{(\sum_j|\zeta_j|^2)^{1/2}\leq1}
\bigg(\sum_{j}\big|b(z)-b_B\big|\big|f_j(z)\big|\cdot\zeta_j\bigg)dz\notag\\
&\leq C\sum_{l=1}^\infty\frac{1}{|2^{l+1}B|}\int_{2^{l+1}B}\big|b(z)-b_B\big|\bigg(\sum_{j}\big|f_j(z)\big|^2\bigg)^{1/2}dz.
\end{align}
Consequently, from the pointwise estimates \eqref{key estimate1} and \eqref{key estimate2}, it follows that
\begin{equation*}
\begin{split}
J_2&\leq\frac{C}{\theta(w(B))^{1/p}}\bigg(\int_B\big|b(x)-b_B\big|^pw(x)\,dx\bigg)^{1/p}
\times\Bigg(\sum_{l=1}^\infty\frac{1}{|2^{l+1}B|}\int_{2^{l+1}B}\bigg(\sum_{j}\big|f_j(z)\big|^2\bigg)^{1/2}\,dz\Bigg)\\
&+C\cdot\frac{w(B)^{1/p}}{\theta(w(B))^{1/p}}
\sum_{l=1}^\infty\frac{1}{|2^{l+1}B|}\int_{2^{l+1}B}\big|b_{2^{l+1}B}-b_B\big|\bigg(\sum_{j}\big|f_j(z)\big|^2\bigg)^{1/2}\,dz\\
&+C\cdot\frac{w(B)^{1/p}}{\theta(w(B))^{1/p}}
\sum_{l=1}^\infty\frac{1}{|2^{l+1}B|}\int_{2^{l+1}B}\big|b(z)-b_{2^{l+1}B}\big|\bigg(\sum_{j}\big|f_j(z)\big|^2\bigg)^{1/2}\,dz\\
&:=J_3+J_4+J_5.
\end{split}
\end{equation*}
Recall that the following estimate holds for $w\in A_p$ and $1\leq p<\infty$:
\begin{equation}\label{BMO2}
\bigg(\int_B\big|b(x)-b_B\big|^pw(x)\,dx\bigg)^{1/p}\leq C\|b\|_*\cdot w(B)^{1/p}.
\end{equation}
Indeed, since $w\in A_p$ with $1\leq p<\infty$, we know that there exists a number $r>1$ such that $w\in RH_r$. By using H\"older's inequality and John--Nirenberg's inequality for $BMO$ functions (see \cite{duoand,john}), we find that
\begin{equation*}
\begin{split}
\bigg(\int_B\big|b(x)-b_B\big|^pw(x)\,dx\bigg)^{1/p}&\leq\left(\int_{B}\big|b(x)-b_{B}\big|^{pr'}dx\right)^{1/{(pr')}}
\left(\int_{B}w(x)^r\,dx\right)^{1/{(pr)}}\\
&\leq C\cdot w(B)^{1/p}\left(\frac{1}{|B|}\int_{B}\big|b(x)-b_{B}\big|^{pr'}dx\right)^{1/{(pr')}}\\
&\leq C\|b\|_*\cdot w(B)^{1/p}.
\end{split}
\end{equation*}
Furthermore, it follows from (\ref{BMO2}), H\"older's inequality and the $A_p$ condition on $w$ that
\begin{equation*}
\begin{split}
J_3&\leq C\|b\|_*\cdot\frac{w(B)^{1/p}}{\theta(w(B))^{1/p}}\sum_{l=1}^\infty\frac{1}{|2^{l+1}B|}\Bigg(\int_{2^{l+1}B}
\bigg(\sum_{j}\big|f_j(z)\big|^2\bigg)^{p/2}w(z)\,dz\Bigg)^{1/p}\\
&\ \times\left(\int_{2^{l+1}B}w(z)^{-{p'}/p}\,dz\right)^{1/{p'}}\\
&\leq C\bigg\|\bigg(\sum_{j}\big|f_j\big|^2\bigg)^{1/2}\bigg\|_{\mathcal M^{p,\theta}(w)}
\times\sum_{l=1}^\infty\frac{\theta(w(2^{l+1}B))^{1/p}}{\theta(w(B))^{1/p}}\cdot\frac{w(B)^{1/p}}{w(2^{l+1}B)^{1/p}}\\
&\leq C\bigg\|\bigg(\sum_{j}\big|f_j\big|^2\bigg)^{1/2}\bigg\|_{\mathcal M^{p,\theta}(w)},
\end{split}
\end{equation*}
where in the last inequality we have used the estimate \eqref{theta1}. For the term $J_4$, since $b\in BMO(\mathbb R^n)$, a simple calculation shows that for every ball $B$~(or cube $Q$)
\begin{equation}\label{BMO}
\big|b_{2^{l+1}B}-b_B\big|\leq C\cdot(l+1)\|b\|_*.
\end{equation}
We then use H\"older's inequality, (\ref{BMO}) and the $A_p$ condition on $w$ to obtain
\begin{equation*}
\begin{split}
J_4&\leq C\|b\|_*\cdot\frac{w(B)^{1/p}}{\theta(w(B))^{1/p}}\sum_{l=1}^\infty\frac{l+1}{|2^{l+1}B|}\Bigg(\int_{2^{l+1}B}
\bigg(\sum_{j}\big|f_j(z)\big|^2\bigg)^{p/2}w(z)\,dz\Bigg)^{1/p}\\
&\ \times\left(\int_{2^{l+1}B}w(z)^{-{p'}/p}\,dz\right)^{1/{p'}}\\
&\leq C\bigg\|\bigg(\sum_{j}\big|f_j\big|^2\bigg)^{1/2}\bigg\|_{\mathcal M^{p,\theta}(w)}
\times\sum_{l=1}^\infty\big(l+1\big)\cdot\frac{\theta(w(2^{l+1}B))^{1/p}}{\theta(w(B))^{1/p}}\cdot\frac{w(B)^{1/p}}{w(2^{l+1}B)^{1/p}}.
\end{split}
\end{equation*}
Applying the $\mathcal D_\kappa$ condition (\ref{D condition}) of $\theta$ and the inequality \eqref{compare} again together with the fact that $0\leq\kappa<1$, we thus have
\begin{align}\label{theta3}
\sum_{l=1}^\infty\big(l+1\big)\cdot\frac{\theta(w(2^{l+1}B))^{1/p}}{\theta(w(B))^{1/p}}\cdot\frac{w(B)^{1/p}}{w(2^{l+1}B)^{1/p}}
&\leq C\sum_{l=1}^\infty\big(l+1\big)\cdot\frac{w(B)^{{(1-\kappa)}/p}}{w(2^{l+1}B)^{{(1-\kappa)}/p}}\notag\\
&\leq C\sum_{l=1}^\infty\big(l+1\big)\cdot\left(\frac{|B|}{|2^{l+1}B|}\right)^{\delta {(1-\kappa)}/p}\notag\\
&\leq C\sum_{l=1}^\infty\big(l+1\big)\cdot\left(\frac{1}{2^{(l+1)n}}\right)^{\delta {(1-\kappa)}/p}\notag\\
&\leq C,
\end{align}
which in turn gives that
\begin{equation*}
J_4\leq C\bigg\|\bigg(\sum_{j}\big|f_j\big|^2\bigg)^{1/2}\bigg\|_{\mathcal M^{p,\theta}(w)}.
\end{equation*}
It remains to estimate the last term $J_5$. An application of H\"older's inequality gives us that
\begin{equation*}
\begin{split}
J_5&\leq C\cdot\frac{w(B)^{1/p}}{\theta(w(B))^{1/p}}\sum_{l=1}^\infty\frac{1}{|2^{l+1}B|}\Bigg(\int_{2^{l+1}B}
\bigg(\sum_{j}\big|f_j(z)\big|^2\bigg)^{p/2}w(z)\,dz\Bigg)^{1/p}\\
&\ \times\left(\int_{2^{l+1}B}\big|b(z)-b_{2^{l+1}B}\big|^{p'}w(z)^{-{p'}/p}\,dz\right)^{1/{p'}}.
\end{split}
\end{equation*}
If we set $\mu(z)=w(z)^{-{p'}/p}$, then we have $\mu\in A_{p'}$ because $w\in A_p$(see \cite{duoand,garcia}). Then it follows from the inequality (\ref{BMO2}) and the $A_p$ condition that
\begin{align}\label{BMO2}
\left(\int_{2^{l+1}B}\big|b(z)-b_{2^{l+1}B}\big|^{p'}\mu(z)\,dz\right)^{1/{p'}}&\leq C\|b\|_*\cdot \mu\big(2^{l+1}B\big)^{1/{p'}}\notag\\
&= C\|b\|_*\cdot\left(\int_{2^{l+1}B}w(z)^{-{p'}/p}dz\right)^{1/{p'}}\notag\\
&\leq C\|b\|_*\cdot\frac{|2^{l+1}B|}{w(2^{l+1}B)^{1/p}}.
\end{align}
Therefore, in view of the estimates \eqref{BMO2} and \eqref{theta1}, we conclude that
\begin{equation*}
\begin{split}
J_5&\leq C\|b\|_*\cdot\frac{w(B)^{1/p}}{\theta(w(B))^{1/p}}\sum_{l=1}^\infty\frac{1}{w(2^{l+1}B)^{1/p}}\Bigg(\int_{2^{l+1}B}
\bigg(\sum_{j}\big|f_j(z)\big|^2\bigg)^{p/2}w(z)\,dz\Bigg)^{1/p}\\
&\leq C\bigg\|\bigg(\sum_{j}\big|f_j\big|^2\bigg)^{1/2}\bigg\|_{\mathcal M^{p,\theta}(w)}
\times\sum_{l=1}^\infty\frac{\theta(w(2^{l+1}B))^{1/p}}{\theta(w(B))^{1/p}}\cdot\frac{w(B)^{1/p}}{w(2^{l+1}B)^{1/p}}\\
&\leq C\bigg\|\bigg(\sum_{j}\big|f_j\big|^2\bigg)^{1/2}\bigg\|_{\mathcal M^{p,\theta}(w)}.
\end{split}
\end{equation*}
Summarizing the above discussions, we finish the proof of the main theorem.
\end{proof}

\section{Proof of Theorem \ref{mainthm:4}}

\begin{proof}
Inspired by the works in \cite{perez2,perez4,zhang}, for any fixed $\sigma>0$, we apply the Calder\'on--Zygmund decomposition of $\vec{f}=(f_1,f_2,\ldots)$ at height $\sigma$ to obtain a collection of disjoint non-overlapping dyadic cubes $\{Q_i\}$ such that the following property holds (see \cite{stein,perez4})
\begin{equation}\label{decomposition}
\sigma<\frac{1}{|Q_i|}\int_{Q_i}\bigg(\sum_{j}\big|f_j(y)\big|^2\bigg)^{1/2}dy\leq 2^n\cdot\sigma,
\end{equation}
where $Q_i=Q(c_i,\ell_i)$ denotes the cube centered at $c_i$ with side length $\ell_i$ and all cubes are
assumed to have their sides parallel to the coordinate axes. If we set $E=\bigcup_i Q_i$, then
\begin{equation*}
\bigg(\sum_{j}\big|f_j(y)\big|^2\bigg)^{1/2}\leq\sigma, \quad \mbox{a.e. }\, x\in\mathbb R^n\backslash E.
\end{equation*}
Now we proceed to construct vector-valued version of the Calder\'on--Zygmund decomposition. Define two vector-valued functions $\vec{g}=(g_1,g_2,\ldots)$ and $\vec{h}=(h_1,h_2,\ldots)$ as follows:
\begin{equation*}
g_j(x)=
\begin{cases}
f_j(x) &  \mbox{if}\;\; x\in E^c,\\
\displaystyle\frac{1}{|Q_i|}\int_{Q_i}f_j(y)\,dy    &  \mbox{if}\;\; x\in Q_i,
\end{cases}
\end{equation*}
and
\begin{equation*}
h_j(x)=f_j(x)-g_j(x)=\sum_i h_{ij}(x),\quad j=1,2,\dots,
\end{equation*}
where $h_{ij}(x)=h_j(x)\cdot\chi_{Q_i}(x)=\big(f_j(x)-g_j(x)\big)\cdot\chi_{Q_i}(x)$. Then we have
\begin{equation}\label{pointwise estimate g}
\bigg(\sum_{j}\big|g_j(x)\big|^2\bigg)^{1/2}\leq C\cdot\sigma, \quad \mbox{a.e. }\, x\in\mathbb R^n,
\end{equation}
and
\begin{equation}\label{f=g+h}
\vec{f}=\vec{g}+\vec{h}:=(g_1+h_1,g_2+h_2,\ldots).
\end{equation}
Obviously, $h_{ij}$ is supported on $Q_i$, $i,j=1,2,\dots$,
\begin{equation*}
\int_{\mathbb R^n}h_{ij}(x)\,dx=0,\quad \mbox{and}\quad \big\|h_{ij}\big\|_{L^1}=\int_{\mathbb R^n}\big|h_{ij}(x)\big|\,dx\leq 2\int_{Q_i}\big|f_j(x)\big|\,dx
\end{equation*}
according to the above decomposition. By \eqref{f=g+h} and Minkowski's inequality,
\begin{equation*}
\bigg(\sum_{j}\big|\big[b,\mathcal S_{\alpha}\big](f_j)(x)\big|^2\bigg)^{1/2}\leq\bigg(\sum_{j}\big|\big[b,\mathcal S_{\alpha}\big](g_j)(x)\big|^2\bigg)^{1/2}+
\bigg(\sum_{j}\big|\big[b,\mathcal S_{\alpha}\big](h_j)(x)\big|^2\bigg)^{1/2}.
\end{equation*}
Then we can write
\begin{equation*}
\begin{split}
&w\bigg(\bigg\{x\in\mathbb R^n:\bigg(\sum_{j}\big|\big[b,\mathcal S_{\alpha}\big](f_j)(x)\big|^2\bigg)^{1/2}>\sigma\bigg\}\bigg)\\
&\leq
w\bigg(\bigg\{x\in\mathbb R^n:\bigg(\sum_{j}\big|\big[b,\mathcal S_{\alpha}\big](g_j)(x)\big|^2\bigg)^{1/2}>\sigma/2\bigg\}\bigg)\\
&+w\bigg(\bigg\{x\in\mathbb R^n:\bigg(\sum_{j}\big|\big[b,\mathcal S_{\alpha}\big](h_j)(x)\big|^2\bigg)^{1/2}>\sigma/2\bigg\}\bigg)\\
&:=K_1+K_2.
\end{split}
\end{equation*}
Observe that $w\in A_1\subset A_2$. Applying Chebyshev's inequality and Theorem \ref{commutator thm}, we obtain
\begin{equation*}
K_1\leq \frac{4}{\sigma^2}\cdot\bigg\|\bigg(\sum_{j}\big|\big[b,\mathcal S_{\alpha}\big](g_j)\big|^2\bigg)^{1/2}\bigg\|^2_{L^2_w}
\leq \frac{C}{\sigma^2}\cdot\bigg\|\bigg(\sum_{j}\big|g_j\big|^2\bigg)^{1/2}\bigg\|^2_{L^2_w}.
\end{equation*}
Moreover, by the inequality (\ref{pointwise estimate g}),
\begin{align*}
&\bigg\|\bigg(\sum_{j}\big|g_j\big|^2\bigg)^{1/2}\bigg\|^2_{L^2_w}\\
&\leq
C\cdot\sigma\int_{\mathbb R^n}\bigg(\sum_{j}\big|g_j(x)\big|^2\bigg)^{1/2}w(x)\,dx\\
&\leq C\cdot\sigma\left(\int_{E^c}\bigg(\sum_{j}\big|f_j(x)\big|^2\bigg)^{1/2}w(x)\,dx
+\int_{\bigcup_i Q_i}\bigg(\sum_{j}\big|g_j(x)\big|^2\bigg)^{1/2}w(x)\,dx\right).
\end{align*}
Recall that $g_j(x)=\displaystyle\frac{1}{|Q_i|}\int_{Q_i}f_j(y)\,dy$ when $x\in Q_i$. As before, by using duality and Cauchy--Schwarz inequality, we can see the following estimate is valid for all $x\in Q_i$.
\begin{equation}\label{vectorg}
\bigg(\sum_{j}\big|g_j(x)\big|^2\bigg)^{1/2}\leq\frac{1}{|Q_i|}\int_{Q_i}\bigg(\sum_{j}\big|f_j(y)\big|^2\bigg)^{1/2}dy.
\end{equation}
This estimate \eqref{vectorg} along with the $A_1$ condition yields
\begin{align}\label{g}
&\bigg\|\bigg(\sum_{j}\big|g_j\big|^2\bigg)^{1/2}\bigg\|^2_{L^2_w}\notag\\
&\le C\cdot\sigma\left(\int_{\mathbb R^n}\bigg(\sum_{j}\big|f_j(x)\big|^2\bigg)^{1/2}w(x)\,dx
+\sum_i\frac{w(Q_i)}{|Q_i|}\int_{Q_i}\bigg(\sum_{j}\big|f_j(y)\big|^2\bigg)^{1/2}dy\right)\notag\\
&\leq C\cdot\sigma\left(\int_{\mathbb R^n}\bigg(\sum_{j}\big|f_j(x)\big|^2\bigg)^{1/2}w(x)\,dx
+\sum_i\underset{y\in Q_i}{\mbox{ess\,inf}}\,w(y)\int_{Q_i}\bigg(\sum_{j}\big|f_j(y)\big|^2\bigg)^{1/2}dy\right)\notag\\
&\leq C\cdot\sigma\left(\int_{\mathbb R^n}\bigg(\sum_{j}\big|f_j(x)\big|^2\bigg)^{1/2}w(x)\,dx
+\int_{\bigcup_i Q_i}\bigg(\sum_{j}\big|f_j(y)\big|^2\bigg)^{1/2}w(y)\,dy\right)\notag\\
&\le C\cdot\sigma\int_{\mathbb R^n}\bigg(\sum_{j}\big|f_j(x)\big|^2\bigg)^{1/2}w(x)\,dx.
\end{align}
So we have
\begin{equation*}
K_1\leq C\int_{\mathbb R^n}\frac{\|\vec{f}(x)\|_{\ell^2}}{\sigma}\cdot w(x)\,dx\leq C\int_{\mathbb R^n}\Phi\bigg(\frac{\|\vec{f}(x)\|_{\ell^2}}{\sigma}\bigg)\cdot w(x)\,dx.
\end{equation*}
To deal with the other term $K_2$, let $Q_i^*=2\sqrt n Q_i$ be the cube concentric with $Q_i$ such that $\ell(Q_i^*)=(2\sqrt n)\ell(Q_i)$. Then we can further decompose $K_2$ as follows.
\begin{equation*}
\begin{split}
K_2\le&\,w\bigg(\bigg\{x\in\bigcup_i Q_i^*:\bigg(\sum_{j}\big|\big[b,\mathcal S_{\alpha}\big](h_j)(x)\big|^2\bigg)^{1/2}>\sigma/2\bigg\}\bigg)\\
&+w\bigg(\bigg\{x\notin \bigcup_i Q_i^*:\bigg(\sum_{j}\big|\big[b,\mathcal S_{\alpha}\big](h_j)(x)\big|^2\bigg)^{1/2}>\sigma/2\bigg\}\bigg)\\
:=&\,K_3+K_4.
\end{split}
\end{equation*}
Since $w\in A_1$, then by the inequality (\ref{weights}), we can get
\begin{equation*}
K_3\leq\sum_i w\big(Q_i^*\big)\le C\sum_i w(Q_i).
\end{equation*}
Furthermore, it follows from the inequality (\ref{decomposition}) and the $A_1$ condition that
\begin{equation*}
\begin{split}
K_3&\leq C\sum_i\frac{\,1\,}{\sigma}\cdot\underset{y\in Q_i}{\mbox{ess\,inf}}\,w(y)
\int_{Q_i}\bigg(\sum_{j}\big|f_j(y)\big|^2\bigg)^{1/2}dy\\
&\leq\frac{C}{\sigma}\sum_i\int_{Q_i}\bigg(\sum_{j}\big|f_j(y)\big|^2\bigg)^{1/2}w(y)\,dy\\
&\leq\frac{C}{\sigma}\int_{\bigcup_i Q_i}\bigg(\sum_{j}\big|f_j(y)\big|^2\bigg)^{1/2}w(y)\,dy\\
\end{split}
\end{equation*}
\begin{equation*}
\begin{split}
&\leq C\int_{\mathbb R^n}\frac{\|\vec{f}(y)\|_{\ell^2}}{\sigma}\cdot w(y)\,dy
\leq C\int_{\mathbb R^n}\Phi\bigg(\frac{\|\vec{f}(y)\|_{\ell^2}}{\sigma}\bigg)\cdot w(y)\,dy.
\end{split}
\end{equation*}
Arguing as in the proof of Theorem \ref{mainthm:3}, for any given $x\in \mathbb R^n$, $(y,t)\in\Gamma(x)$ and for $j=1,2,\dots$, we also find that
\begin{align*}
&\sup_{\varphi\in{\mathcal C}_\alpha}\bigg|\int_{\mathbb R^n}\big[b(x)-b(z)\big]\varphi_t(y-z)\sum_i h_{ij}(z)\,dz\bigg|\notag\\
&\leq\sup_{\varphi\in{\mathcal C}_\alpha}\bigg|\sum_i \big[b(x)-b_{Q_i}\big]\int_{\mathbb R^n}\varphi_t(y-z)h_{ij}(z)\,dz\bigg|\notag\\
&+\sup_{\varphi\in{\mathcal C}_\alpha}\bigg|\int_{\mathbb R^n}\varphi_t(y-z)\sum_i \big[b_{Q_i}-b(z)\big]h_{ij}(z)\,dz\bigg|\\
&\leq\sum_i\big|b(x)-b_{Q_i}\big|\cdot\sup_{\varphi\in{\mathcal C}_\alpha}\bigg|\int_{\mathbb R^n}\varphi_t(y-z)h_{ij}(z)\,dz\bigg|\notag\\
&+\sup_{\varphi\in{\mathcal C}_\alpha}\bigg|\int_{\mathbb R^n}\varphi_t(y-z)\sum_i \big[b_{Q_i}-b(z)\big]h_{ij}(z)\,dz\bigg|.
\end{align*}
Hence, by definition, we have that for any given $x\in \mathbb R^n$ and $j\in\mathbb Z^+$,
\begin{equation}\label{commutator estimate}
\begin{split}
\big|\big[b,\mathcal S_\alpha\big](h_j)(x)\big|&\leq\sum_i\big|b(x)-b_{Q_i}\big|\cdot\mathcal S_\alpha(h_{ij})(x)+\mathcal S_\alpha\bigg(\sum_i[b_{Q_i}-b]h_{ij}\bigg)(x).
\end{split}
\end{equation}
On the other hand, by duality argument and Cauchy--Schwarz inequality, we can see the following vector-valued form of Minkowski's inequality is true for any real numbers $\nu_{ij}\in\mathbb R$, $i,j=1,2,\dots$.
\begin{equation}\label{min}
\bigg(\sum_{j}\bigg|\sum_i\big|\nu_{ij}\big|\bigg|^2\bigg)^{1/2}\leq\sum_i\bigg(\sum_{j}\big|\nu_{ij}\big|^2\bigg)^{1/2}.
\end{equation}
Therefore, by the estimates (\ref{commutator estimate}) and (\ref{min}), we get
\begin{equation*}
\begin{split}
\bigg(\sum_{j}\big|\big[b,\mathcal S_{\alpha}\big](h_j)(x)\big|^2\bigg)^{1/2}
&\leq\bigg(\sum_{j}\bigg|\sum_i\big|b(x)-b_{Q_i}\big|\cdot\mathcal S_\alpha(h_{ij})(x)\bigg|^2\bigg)^{1/2}\\
&+\bigg(\sum_{j}\bigg|\mathcal S_\alpha\bigg(\sum_i[b_{Q_i}-b]h_{ij}\bigg)(x)\bigg|^2\bigg)^{1/2}\\
&\leq\sum_i\big|b(x)-b_{Q_i}\big|\cdot\bigg(\sum_{j}\big|\mathcal S_\alpha(h_{ij})(x)\big|^2\bigg)^{1/2}\\
&+\bigg(\sum_{j}\bigg|\mathcal S_\alpha\bigg(\sum_i[b_{Q_i}-b]h_{ij}\bigg)(x)\bigg|^2\bigg)^{1/2}.
\end{split}
\end{equation*}
Then the term $K_4$ can be divided into two parts.
\begin{equation*}
\begin{split}
K_4\leq &w\bigg(\bigg\{x\notin \bigcup_i Q_i^*:
\sum_i\big|b(x)-b_{Q_i}\big|\cdot\bigg(\sum_{j}\big|\mathcal S_\alpha(h_{ij})(x)\big|^2\bigg)^{1/2}>\sigma/4\bigg\}\bigg)\\
&+w\bigg(\bigg\{x\notin \bigcup_i Q_i^*:
\bigg(\sum_{j}\bigg|\mathcal S_\alpha\bigg(\sum_i[b_{Q_i}-b]h_{ij}\bigg)(x)\bigg|^2\bigg)^{1/2}>\sigma/4\bigg\}\bigg)\\
:=&K_5+K_6.
\end{split}
\end{equation*}
It follows directly from the Chebyshev's inequality that
\begin{equation*}
\begin{split}
K_5&\leq\frac{\,4\,}{\sigma}\int_{\mathbb R^n\backslash\bigcup_i Q_i^*}\sum_i\big|b(x)-b_{Q_i}\big|\cdot
\bigg(\sum_{j}\big|\mathcal S_\alpha(h_{ij})(x)\big|^2\bigg)^{1/2}w(x)\,dx\\
&\leq\frac{\,4\,}{\sigma}\sum_i\left(\int_{(Q_i^*)^c}\big|b(x)-b_{Q_i}\big|\cdot
\bigg(\sum_{j}\big|\mathcal S_\alpha(h_{ij})(x)\big|^2\bigg)^{1/2}w(x)\,dx\right).\\
\end{split}
\end{equation*}
Denote by $c_i$ the center of $Q_i$. For any $\varphi\in{\mathcal C}_\alpha$, $0<\alpha\le1$, by the cancellation condition of $h_{ij}$ over $Q_i$, we obtain that for any $(y,t)\in\Gamma(x)$ and for $i,j=1,2,\dots$,
\begin{align}\label{kernel estimate1}
\big|(\varphi_t*h_{ij})(y)\big|&=\left|\int_{Q_i}\big[\varphi_t(y-z)-\varphi_t(y-c_i)\big]h_{ij}(z)\,dz\right|\notag\\
&\leq\int_{Q_i\cap\{z:|z-y|\le t\}}\frac{|z-c_i|^\alpha}{t^{n+\alpha}}\big|h_{ij}(z)\big|\,dz\notag\\
&\leq C\cdot\frac{\ell(Q_i)^{\alpha}}{t^{n+\alpha}}\int_{Q_i\cap\{z:|z-y|\le t\}}\big|h_{ij}(z)\big|\,dz.
\end{align}
In addition, for any $z\in Q_i$ and $x\in (Q^*_i)^c$, we have $|z-c_i|<\frac{|x-c_i|}{2}$. Thus, for all $(y,t)\in\Gamma(x)$ and $|z-y|\le t$ with $z\in Q_i$, it is easy to see that
\begin{equation}\label{2t}
t+t\ge|x-y|+|y-z|\ge|x-z|\ge|x-c_i|-|z-c_i|\ge\frac{|x-c_i|}{2}.
\end{equation}
Hence, for any $x\in (Q^*_i)^c$, by using the above inequalities (\ref{kernel estimate1}) and (\ref{2t}) along with the fact that $ \big\|h_{ij}\big\|_{L^1}\leq 2\int_{Q_i}\big|f_j(x)\big|\,dx$, we obtain that for any $i,j=1,2,\dots$,
\begin{equation*}
\begin{split}
\big|\mathcal S_{\alpha}(h_{ij})(x)\big|&=\left(\iint_{\Gamma(x)}
\bigg(\sup_{\varphi\in{\mathcal C}_\alpha}\big|(\varphi_t*{h_{ij}})(y)\big|\bigg)^2\frac{dydt}{t^{n+1}}\right)^{1/2}\\
&\leq C\cdot\ell(Q_i)^{\alpha}\bigg(\int_{Q_i}\big|h_{ij}(z)\big|\,dz\bigg)
\left(\int_{\frac{|x-c_i|}{4}}^\infty\int_{|y-x|<t}\frac{dydt}{t^{2(n+\alpha)+n+1}}\right)^{1/2}\\
&\leq C\cdot\ell(Q_i)^{\alpha}\bigg(\int_{Q_i}\big|h_{ij}(z)\big|\,dz\bigg)
\left(\int_{\frac{|x-c_i|}{4}}^\infty\frac{dt}{t^{2(n+\alpha)+1}}\right)^{1/2}\\
\end{split}
\end{equation*}
\begin{equation*}
\begin{split}
&\leq C\cdot\frac{\ell(Q_i)^{\alpha}}{|x-c_i|^{n+\alpha}}\bigg(\int_{Q_i}\big|f_j(z)\big|\,dz\bigg).
\end{split}
\end{equation*}
Furthermore, by duality and Cauchy--Schwarz inequality again,
\begin{equation*}
\bigg(\sum_{j}\big|\mathcal S_\alpha(h_{ij})(x)\big|^2\bigg)^{1/2}\leq
C\cdot\frac{\ell(Q_i)^{\alpha}}{|x-c_i|^{n+\alpha}}\times\int_{Q_i}\bigg(\sum_{j}\big|f_j(z)\big|^2\bigg)^{1/2}dz.
\end{equation*}
Since $Q_i^*=2\sqrt n Q_i\supset 2Q_i$, then $(Q_i^*)^c\subset (2Q_i)^c$. This fact together with the above pointwise estimate yields
\begin{equation*}
\begin{split}
K_5&\leq\frac{C}{\sigma}\sum_i
\left(\ell(Q_i)^{\alpha}\int_{Q_i}\bigg(\sum_{j}\big|f_j(z)\big|^2\bigg)^{1/2}dz
\times\int_{(Q_i^*)^c}\big|b(x)-b_{Q_i}\big|\cdot\frac{w(x)}{|x-c_i|^{n+\alpha}}dx\right)\\
&\leq\frac{C}{\sigma}\sum_i
\left(\ell(Q_i)^{\alpha}\int_{Q_i}\bigg(\sum_{j}\big|f_j(z)\big|^2\bigg)^{1/2}dz
\times\int_{(2Q_i)^c}\big|b(x)-b_{Q_i}\big|\cdot\frac{w(x)}{|x-c_i|^{n+\alpha}}dx\right)\\
&\leq\frac{C}{\sigma}\sum_i
\left(\ell(Q_i)^{\alpha}\int_{Q_i}\bigg(\sum_{j}\big|f_j(z)\big|^2\bigg)^{1/2}dz
\times\sum_{l=1}^\infty\int_{2^{l+1}Q_i\backslash 2^{l}Q_i}\big|b(x)-b_{2^{l+1}Q_i}\big|\cdot\frac{w(x)}{|x-c_i|^{n+\alpha}}dx\right)\\
&+\frac{C}{\sigma}\sum_i
\left(\ell(Q_i)^{\alpha}\int_{Q_i}\bigg(\sum_{j}\big|f_j(z)\big|^2\bigg)^{1/2}dz\times\sum_{l=1}^\infty\int_{2^{l+1}Q_i\backslash 2^{l}Q_i}\big|b_{2^{l+1}Q_i}-b_{Q_i}\big|\cdot\frac{w(x)}{|x-c_i|^{n+\alpha}}dx\right)\\
&:=\mbox{\upshape I+II}.
\end{split}
\end{equation*}
For the term I, it then follows from (\ref{BMO2})(consider $2^{l+1}Q_i$ instead of $B$), (\ref{general weights}) and the fact that $w\in A_1$,
\begin{equation*}
\begin{split}
\mbox{\upshape I}&\leq\frac{C}{\sigma}\sum_i
\left(\ell(Q_i)^{\alpha}\int_{Q_i}\bigg(\sum_{j}\big|f_j(z)\big|^2\bigg)^{1/2}dz\times\sum_{l=1}^\infty\frac{1}{[2^{l-1}\ell(Q_i)]^{n+\alpha}}
\int_{2^{l+1}Q_i}\big|b(x)-b_{2^{l+1}Q_i}\big|\cdot w(x)\,dx\right)\\
&\leq\frac{C\cdot\|b\|_*}{\sigma}\sum_i\left(\int_{Q_i}\bigg(\sum_{j}\big|f_j(z)\big|^2\bigg)^{1/2}dz
\times\sum_{l=1}^\infty\frac{w\big(2^{l+1}Q_i\big)}{(2^{l-1})^{n+\alpha}|Q_i|}\right)\\
&\leq\frac{C\cdot\|b\|_*}{\sigma}\sum_i\left(\int_{Q_i}\bigg(\sum_{j}\big|f_j(z)\big|^2\bigg)^{1/2}dz
\times\sum_{l=1}^\infty\frac{(2^{l+1})^nw\big(Q_i\big)}{(2^{l-1})^{n+\alpha}|Q_i|}\right)\\
&\leq\frac{C}{\sigma}\sum_i\left(\frac{w\big(Q_i\big)}{|Q_i|}\cdot\int_{Q_i}\bigg(\sum_{j}\big|f_j(z)\big|^2\bigg)^{1/2}dz
\times\sum_{l=1}^\infty\frac{1}{2^{l\alpha}}\right)\\
&\leq\frac{C}{\sigma}\sum_i\underset{z\in Q_i}{\mbox{ess\,inf}}\,w(z)\int_{Q_i}\bigg(\sum_{j}\big|f_j(z)\big|^2\bigg)^{1/2}dz\\
\end{split}
\end{equation*}
\begin{equation*}
\begin{split}
&\leq \frac{C}{\sigma}\int_{\bigcup_i Q_i}\bigg(\sum_{j}\big|f_j(z)\big|^2\bigg)^{1/2}w(z)\,dz\\
&\leq C\int_{\mathbb R^n}\frac{\|\vec{f}(z)\|_{\ell^2}}{\sigma}\cdot w(z)\,dz
\leq C\int_{\mathbb R^n}\Phi\bigg(\frac{\|\vec{f}(z)\|_{\ell^2}}{\sigma}\bigg)\cdot w(z)\,dz.
\end{split}
\end{equation*}
For the term II, from the inequalities (\ref{BMO}) and (\ref{general weights}) along with the fact that $w\in A_1$, it then follows that
\begin{equation*}
\begin{split}
\mbox{\upshape II}&\leq\frac{C\cdot\|b\|_*}{\sigma}\sum_i\left(\ell(Q_i)^{\alpha}\int_{Q_i}\bigg(\sum_{j}\big|f_j(z)\big|^2\bigg)^{1/2}dz
\times\sum_{l=1}^\infty\big(l+1\big)\cdot\frac{w\big(2^{l+1}Q_i\big)}{[2^{l-1}\ell(Q_i)]^{n+\alpha}}\right)\\
&\leq\frac{C\cdot\|b\|_*}{\sigma}\sum_i\left(\int_{Q_i}\bigg(\sum_{j}\big|f_j(z)\big|^2\bigg)^{1/2}dz
\times\sum_{l=1}^\infty\big(l+1\big)\cdot\frac{(2^{l+1})^nw\big(Q_i\big)}{(2^{l-1})^{n+\alpha}|Q_i|}\right)\\
&\leq\frac{C}{\sigma}\sum_i\left(\frac{w\big(Q_i\big)}{|Q_i|}\cdot\int_{Q_i}\bigg(\sum_{j}\big|f_j(z)\big|^2\bigg)^{1/2}dz
\times\sum_{l=1}^\infty\frac{l+1}{2^{l\alpha}}\right)\\
&\leq\frac{C}{\sigma}\sum_i\left(\frac{w\big(Q_i\big)}{|Q_i|}\cdot\int_{Q_i}\bigg(\sum_{j}\big|f_j(z)\big|^2\bigg)^{1/2}dz\right)
\leq C\int_{\mathbb R^n}\Phi\bigg(\frac{\|\vec{f}(z)\|_{\ell^2}}{\sigma}\bigg)\cdot w(z)\,dz.
\end{split}
\end{equation*}
On the other hand, by using the weighted weak-type (1,1) estimate of vector-valued intrinsic square functions (see Theorem B) and (\ref{min}), we have
\begin{equation*}
\begin{split}
K_6&\leq\frac{C}{\sigma}\int_{\mathbb R^n}\bigg(\sum_{j}\bigg|\sum_i\big|b(x)-b_{Q_i}\big|\big|h_{ij}(x)\big|\bigg|^2\bigg)^{1/2}w(x)\,dx\\
&\leq\frac{C}{\sigma}\int_{\mathbb R^n}\sum_i\big|b(x)-b_{Q_i}\big|\bigg(\sum_{j}\big|h_{ij}(x)\big|^2\bigg)^{1/2}w(x)\,dx\\
&=\frac{C}{\sigma}\sum_i\int_{Q_i}\big|b(x)-b_{Q_i}\big| \bigg(\sum_{j}\big|h_{ij}(x)\big|^2\bigg)^{1/2}w(x)\,dx\\
&\leq\frac{C}{\sigma}\sum_i\int_{Q_i}\big|b(x)-b_{Q_i}\big|\bigg(\sum_{j}\big|f_j(x)\big|^2\bigg)^{1/2}w(x)\,dx\\
&+\frac{C}{\sigma}\sum_i\frac{1}{|Q_i|}\int_{Q_i}\bigg(\sum_{j}\big|f_j(y)\big|^2\bigg)^{1/2}dy\times\int_{Q_i}\big|b(x)-b_{Q_i}\big|w(x)\,dx\\
&:=\mbox{\upshape III+IV}.
\end{split}
\end{equation*}
For the term III, by the generalized H\"older's inequality with weight (\ref{weighted holder}), (\ref{weighted exp}) and (\ref{equiv norm with weight}), we can deduce that
\begin{equation*}
\begin{split}
\mbox{\upshape III}=&\frac{C}{\sigma}\sum_i w(Q_i)\cdot\frac{1}{w(Q_i)}\int_{Q_i}\big|b(x)-b_{Q_i}\big|\bigg(\sum_{j}\big|f_j(x)\big|^2\bigg)^{1/2}w(x)\,dx\\
\end{split}
\end{equation*}
\begin{equation*}
\begin{split}
\leq&\frac{C}{\sigma}\sum_i w(Q_i)\cdot\big\|b-b_{Q_i}\big\|_{\exp L(w),Q_i}
\bigg\|\bigg(\sum_{j}\big|f_j\big|^2\bigg)^{1/2}\bigg\|_{L\log L(w),Q_i}\\
\leq&\frac{C\cdot\|b\|_*}{\sigma}\sum_i w(Q_i)\cdot
\bigg\|\bigg(\sum_{j}\big|f_j\big|^2\bigg)^{1/2}\bigg\|_{L\log L(w),Q_i}\\
\leq&\frac{C\cdot\|b\|_*}{\sigma}\sum_i w(Q_i)\cdot
\inf_{\eta>0}\left\{\eta+\frac{\eta}{w(Q_i)}\int_{Q_i}\Phi\bigg(\frac{\|\vec{f}(y)\|_{\ell^2}}{\eta}\bigg)\cdot w(y)\,dy\right\}\\
\leq&\frac{C\cdot\|b\|_*}{\sigma}\sum_i w(Q_i)\cdot
\left\{\sigma+\frac{\sigma}{w(Q_i)}\int_{Q_i}\Phi\bigg(\frac{\|\vec{f}(y)\|_{\ell^2}}{\sigma}\bigg)\cdot w(y)\,dy\right\}\\
\leq& C\left\{\sum_i w(Q_i)+\sum_i\int_{Q_i}\Phi\bigg(\frac{\|\vec{f}(y)\|_{\ell^2}}{\sigma}\bigg)\cdot w(y)\,dy\right\}\\
\leq& C\int_{\mathbb R^n}\Phi\bigg(\frac{\|\vec{f}(y)\|_{\ell^2}}{\sigma}\bigg)\cdot w(y)\,dy.
\end{split}
\end{equation*}
For the term IV, by the inequality (\ref{BMO2})(consider $Q_i$ instead of $B$) and the fact that $w\in A_1$, we conclude that
\begin{equation*}
\begin{split}
\mbox{\upshape IV}\leq&\frac{C\cdot\|b\|_*}{\sigma}\sum_i\frac{w(Q_i)}{|Q_i|}\int_{Q_i}\bigg(\sum_{j}\big|f_j(y)\big|^2\bigg)^{1/2}dy\\
\leq&\frac{C\cdot\|b\|_*}{\sigma}\sum_i\int_{Q_i}\bigg(\sum_{j}\big|f_j(y)\big|^2\bigg)^{1/2}w(y)\,dy\\
\leq&\frac{C\cdot\|b\|_*}{\sigma}\int_{\bigcup_i Q_i}\bigg(\sum_{j}\big|f_j(y)\big|^2\bigg)^{1/2}w(y)\,dy\\
\leq& C\int_{\mathbb R^n}\frac{\|\vec{f}(y)\|_{\ell^2}}{\sigma}\cdot w(y)\,dy
\leq C\int_{\mathbb R^n}\Phi\bigg(\frac{\|\vec{f}(y)\|_{\ell^2}}{\sigma}\bigg)\cdot w(y)\,dy.
\end{split}
\end{equation*}
Summing up all the above estimates, we get the desired result.
\end{proof}

\section{Proof of Theorem \ref{mainthm:5}}

\begin{proof}
Fix a ball $B=B(x_0,r_B)\subseteq\mathbb R^n$ and decompose $f_j=f^0_j+f^\infty_j$, where $f^0_j=f_j\cdot\chi_{2B}$ and $\chi_{2B}$ denotes the characteristic function of $2B=B(x_0,2r_B)$, $j=1,2,\dots$. For any $0\leq\kappa<1$, $w\in A_1$ and any given $\sigma>0$, we then write
\begin{equation*}
\begin{split}
&\frac{1}{\theta(w(B))}\cdot w\bigg(\bigg\{x\in B:\bigg(\sum_{j}\big|\big[b,\mathcal S_{\alpha}\big](f_j)(x)\big|^2\bigg)^{1/2}>\sigma\bigg\}\bigg)\\
\leq &\frac{1}{\theta(w(B))}\cdot w\bigg(\bigg\{x\in B:\bigg(\sum_{j}\big|\big[b,\mathcal S_{\alpha}\big](f^0_j)(x)\big|^2\bigg)^{1/2}>\sigma/2\bigg\}\bigg)\\
\end{split}
\end{equation*}
\begin{equation*}
\begin{split}
&+\frac{1}{\theta(w(B))}\cdot w\bigg(\bigg\{x\in B:\bigg(\sum_{j}\big|\big[b,\mathcal S_{\alpha}\big](f^{\infty}_j)(x)\big|^2\bigg)^{1/2}>\sigma/2\bigg\}\bigg)\\
:=&K'_1+K'_2.
\end{split}
\end{equation*}
By using Theorem \ref{mainthm:4}, we get
\begin{equation*}
\begin{split}
K'_1&\leq C\cdot\frac{1}{\theta(w(B))}\int_{\mathbb R^n}
\Phi\bigg(\frac{\,1\,}{\sigma}\Big(\sum_{j}|f^0_j(x)|^2\Big)^{1/2}\bigg)\cdot w(x)\,dx\\
&= C\cdot\frac{1}{\theta(w(B))}
\int_{2B}\Phi\bigg(\frac{\,1\,}{\sigma}\Big(\sum_{j}|f_j(x)|^2\Big)^{1/2}\bigg)\cdot w(x)\,dx\\
&= C\cdot\frac{\theta(w(2B))}{\theta(w(B))}\cdot\frac{1}{\theta(w(2B))}
\int_{2B}\Phi\bigg(\frac{\,1\,}{\sigma}\Big(\sum_{j}|f_j(x)|^2\Big)^{1/2}\bigg)\cdot w(x)\,dx.
\end{split}
\end{equation*}
Observe that
\begin{equation}\label{fact}
\frac{1}{\theta(w(B))}\leq\frac{1+\log^+\big(\frac{w(B)}{\theta(w(B))}\big)}{\theta(w(B))}=\frac{\Phi\big(\frac{w(B)}{\theta(w(B))}\big)}{w(B)}.
\end{equation}
Moreover, since $0<w(B)<w(2B)<+\infty$ when $w\in A_1$, then by the $\mathcal D_\kappa$ condition (\ref{D condition}) of $\theta$, the inequality (\ref{weights}) and the fact \eqref{fact}, we have
\begin{equation*}
\begin{split}
K'_1&\leq C\cdot\frac{w(2B)^\kappa}{w(B)^\kappa}\cdot\frac{1}{\theta(w(2B))}
\int_{2B}\Phi\bigg(\frac{\|\vec{f}(x)\|_{\ell^2}}{\sigma}\bigg)\cdot w(x)\,dx\\
&\leq C\cdot\sup_B\left\{\frac{1}{\theta(w(B))}
\int_{B}\Phi\bigg(\frac{\|\vec{f}(x)\|_{\ell^2}}{\sigma}\bigg)\cdot w(x)\,dx\right\}\\
&\leq C\cdot\sup_B\left\{\frac{\Phi\Big(\frac{w(B)}{\theta(w(B))}\Big)}{w(B)}
\int_{B}\Phi\bigg(\frac{\|\vec{f}(x)\|_{\ell^2}}{\sigma}\bigg)\cdot w(x)\,dx\right\}.
\end{split}
\end{equation*}
Recall that the following pointwise estimate holds for any $x\in B$,
\begin{equation*}
\begin{split}
\bigg(\sum_{j}\big|\big[b,\mathcal S_\alpha\big](f^\infty_j)(x)\big|^2\bigg)^{1/2}
&\leq\big|b(x)-b_B\big|\bigg(\sum_{j}\big|\mathcal S_\alpha(f^{\infty}_j)(x)\big|^2\bigg)^{1/2}\\
&+\bigg(\sum_{j}\Big|\mathcal S_\alpha\Big([b_{B}-b]f^\infty_j\Big)(x)\Big|^2\bigg)^{1/2}.
\end{split}
\end{equation*}
So we can divide the term $K'_2$ into two parts.
\begin{equation*}
\begin{split}
K'_2\leq&\frac{1}{\theta(w(B))}\cdot w\bigg(\bigg\{x\in B:\big|b(x)-b_{B}\big|\cdot
\bigg(\sum_{j}\big|\mathcal S_{\alpha}(f^{\infty}_j)(x)\big|^2\bigg)^{1/2}>\sigma/4\bigg\}\bigg)\\
&+\frac{1}{\theta(w(B))}\cdot w\bigg(\bigg\{x\in B:\bigg(\sum_{j}\Big|\mathcal S_\alpha\Big([b_{B}-b]f^\infty_j\Big)(x)\Big|^2\bigg)^{1/2}>\sigma/4\bigg\}\bigg)\\
:=&K'_3+K'_4.
\end{split}
\end{equation*}
Since $w\in A_1$, then there exists a number $r>1$ such that $w\in RH_r$. Hence, by using the previous pointwise estimate (\ref{key estimate1}), Chebyshev's inequality together with H\"older's inequality and John--Nirenberg's inequality (see \cite{john}), we conclude that
\begin{equation*}
\begin{split}
K'_3&\leq\frac{1}{\theta(w(B))}\cdot\frac{\,4\,}{\sigma}\int_B\big|b(x)-b_{B}\big|\cdot
\bigg(\sum_{j}\big|\mathcal S_{\alpha}(f^{\infty}_j)(x)\big|^2\bigg)^{1/2}w(x)\,dx\\
&\leq C\sum_{l=1}^\infty\frac{1}{|2^{l+1}B|}\int_{2^{l+1}B}\frac{\,1\,}{\sigma}\bigg(\sum_{j}\big|f_j(z)\big|^2\bigg)^{1/2}dz\\
&\times\frac{1}{\theta(w(B))}\cdot\left(\int_{B}\big|b(x)-b_{B}\big|^{r'}dx\right)^{1/{r'}}\left(\int_{B}w(x)^rdx\right)^{1/r}\\
&\leq C\sum_{l=1}^\infty\frac{1}{|2^{l+1}B|}\int_{2^{l+1}B}\frac{\|\vec{f}(z)\|_{\ell^2}}{\sigma}\,dz\times\frac{w(B)}{\theta(w(B))}.\\
\end{split}
\end{equation*}
It then follows from the $A_1$ condition and the fact \eqref{fact} that
\begin{equation*}
\begin{split}
K'_3&\leq C\sum_{l=1}^\infty\frac{1}{w(2^{l+1}B)}\int_{2^{l+1}B}\frac{\|\vec{f}(z)\|_{\ell^2}}{\sigma}\cdot w(z)\,dz\times\frac{w(B)}{\theta(w(B))}\\
&=C\sum_{l=1}^\infty\frac{1}{\theta(w(2^{l+1}B))}\cdot\frac{\theta(w(2^{l+1}B))}{w(2^{l+1}B)}\int_{2^{l+1}B}\frac{\|\vec{f}(z)\|_{\ell^2}}{\sigma}\cdot w(z)\,dz\times\frac{w(B)}{\theta(w(B))}\\
&\leq C\cdot\sup_B\left\{\frac{\Phi\Big(\frac{w(B)}{\theta(w(B))}\Big)}{w(B)}
\int_{B}\Phi\bigg(\frac{\|\vec{f}(z)\|_{\ell^2}}{\sigma}\bigg)\cdot w(z)\,dz\right\}\\
&\times\sum_{l=1}^\infty\frac{\theta(w(2^{l+1}B))}{\theta(w(B))}\cdot\frac{w(B)}{w(2^{l+1}B)}.
\end{split}
\end{equation*}
Substituting the previous inequality \eqref{theta2} into the term $K'_3$, we thus obtain
\begin{equation*}
K'_3\leq C\cdot\sup_B\left\{\frac{\Phi\Big(\frac{w(B)}{\theta(w(B))}\Big)}{w(B)}
\int_{B}\Phi\bigg(\frac{\|\vec{f}(z)\|_{\ell^2}}{\sigma}\bigg)\cdot w(z)\,dz\right\}.
\end{equation*}
On the other hand, applying the previous pointwise estimate (\ref{key estimate2}) and Chebyshev's inequality, we have
\begin{equation*}
\begin{split}
K'_4&\leq\frac{1}{\theta(w(B))}\cdot\frac{\,4\,}{\sigma}\int_B
\bigg(\sum_{j}\Big|\mathcal S_\alpha\Big([b_{B}-b]f^\infty_j\Big)(x)\Big|^2\bigg)^{1/2}w(x)\,dx\\
&\leq\frac{w(B)}{\theta(w(B))}\cdot\frac{C}{\sigma}\sum_{l=1}^\infty\frac{1}{|2^{l+1}B|}\int_{2^{l+1}B}
\big|b(z)-b_{B}\big|\cdot\bigg(\sum_{j}\big|f_j(z)\big|^2\bigg)^{1/2}dz\\
\end{split}
\end{equation*}
\begin{equation*}
\begin{split}
&\leq\frac{w(B)}{\theta(w(B))}\cdot\frac{C}{\sigma}\sum_{l=1}^\infty\frac{1}{|2^{l+1}B|}\int_{2^{l+1}B}
\big|b(z)-b_{2^{l+1}B}\big|\cdot\bigg(\sum_{j}\big|f_j(z)\big|^2\bigg)^{1/2}dz\\
&+\frac{w(B)}{\theta(w(B))}\cdot\frac{C}{\sigma}\sum_{l=1}^\infty\frac{1}{|2^{l+1}B|}\int_{2^{l+1}B}
\big|b_{2^{l+1}B}-b_B\big|\cdot\bigg(\sum_{j}\big|f_j(z)\big|^2\bigg)^{1/2}dz\\
&:=K'_5+K'_6.
\end{split}
\end{equation*}
For the term $K'_5$, we first use the generalized H\"older's inequality with weight (\ref{weighted holder}), (\ref{weighted exp}) and (\ref{equiv norm with weight}) together with the $A_1$ condition to obtain
\begin{equation*}
\begin{split}
K'_5&\leq \frac{C}{\sigma}\cdot\frac{w(B)}{\theta(w(B))}\sum_{l=1}^\infty\frac{1}{w(2^{l+1}B)}\int_{2^{l+1}B}
\big|b(z)-b_{2^{l+1}B}\big|\cdot\bigg(\sum_{j}\big|f_j(z)\big|^2\bigg)^{1/2}w(z)\,dz\\
&\leq \frac{C}{\sigma}\cdot\frac{w(B)}{\theta(w(B))}\sum_{l=1}^\infty\big\|b-b_{2^{l+1}B}\big\|_{\exp L(w),2^{l+1}B}
\bigg\|\bigg(\sum_{j}\big|f_j\big|^2\bigg)^{1/2}\bigg\|_{L\log L(w),2^{l+1}B}\\
&\leq \frac{C\|b\|_*}{\sigma}\cdot\frac{w(B)}{\theta(w(B))}\sum_{l=1}^\infty
\inf_{\eta>0}\left\{\eta+\frac{\eta}{w(2^{l+1}B)}\int_{2^{l+1}B}\Phi\bigg(\frac{\|\vec{f}(z)\|_{\ell^2}}{\eta}\bigg)\cdot w(z)\,dz\right\}.\\
\end{split}
\end{equation*}
Moreover, notice that the inequality $\Phi(a\cdot b)\leq\Phi(a)\cdot\Phi(b)$ holds for any $a,b>0$, when $\Phi(t)=t\cdot(1+\log^+t)$. For $l=1,2,\dots$, we may choose $\eta=\displaystyle\frac{\sigma\cdot\theta(w(2^{l+1}B))}{w(2^{l+1}B)}$ and then use the estimate \eqref{theta2} to obtain
\begin{equation*}
\begin{split}
K'_5&\leq \frac{C\|b\|_*}{\sigma}\cdot\frac{w(B)}{\theta(w(B))}\\
\times&\sum_{l=1}^\infty
\left\{\sigma\cdot\frac{\theta(w(2^{l+1}B))}{w(2^{l+1}B)}+
\frac{\sigma}{w(2^{l+1}B)}\cdot\frac{\theta(w(2^{l+1}B))}{w(2^{l+1}B)}\cdot\Phi\bigg(\frac{w(2^{l+1}B)}{\theta(w(2^{l+1}B))}\bigg)
\int_{2^{l+1}B}\Phi\bigg(\frac{\|\vec{f}(z)\|_{\ell^2}}{\sigma}\bigg)\cdot w(z)\,dz\right\}\\
&\leq C\|b\|_*\cdot\left[1+\sup_B\left\{\frac{\Phi\Big(\frac{w(B)}{\theta(w(B))}\Big)}{w(B)}
\int_{B}\Phi\bigg(\frac{\|\vec{f}(z)\|_{\ell^2}}{\sigma}\bigg)\cdot w(z)\,dz\right\}\right]\\
&\times\sum_{l=1}^\infty\frac{\theta(w(2^{l+1}B))}{\theta(w(B))}\cdot\frac{w(B)}{w(2^{l+1}B)}\\
&\leq C\cdot\sup_B\left\{\frac{\Phi\Big(\frac{w(B)}{\theta(w(B))}\Big)}{w(B)}
\int_{B}\Phi\bigg(\frac{\|\vec{f}(z)\|_{\ell^2}}{\sigma}\bigg)\cdot w(z)\,dz\right\}.
\end{split}
\end{equation*}
For the last term $K'_6$ we proceed as follows. An application of the inequality (\ref{BMO}) leads to that
\begin{equation*}
\begin{split}
K'_6&\leq C\cdot \frac{w(B)}{\theta(w(B))}\sum_{l=1}^\infty(l+1)\|b\|_*\cdot
\frac{1}{|2^{l+1}B|}\int_{2^{l+1}B}\frac{\|\vec{f}(z)\|_{\ell^2}}{\sigma}\,dz\\
&\leq C\cdot \frac{w(B)}{\theta(w(B))}\sum_{l=1}^\infty(l+1)\|b\|_*\cdot
\frac{1}{w(2^{l+1}B)}\int_{2^{l+1}B}\frac{\|\vec{f}(z)\|_{\ell^2}}{\sigma}\cdot w(z)\,dz\\
\end{split}
\end{equation*}
\begin{equation*}
\begin{split}
&\leq C\cdot\sup_B\left\{\frac{\Phi\Big(\frac{w(B)}{\theta(w(B))}\Big)}{w(B)}
\int_{B}\Phi\bigg(\frac{\|\vec{f}(z)\|_{\ell^2}}{\sigma}\bigg)\cdot w(z)\,dz\right\}\\
&\times\sum_{l=1}^\infty(l+1)\cdot\frac{\theta(w(2^{l+1}B))}{\theta(w(B))}\cdot\frac{w(B)}{w(2^{l+1}B)}.
\end{split}
\end{equation*}
Notice that $w\in A_1\subset A_\infty$, by using the $\mathcal D_\kappa$ condition (\ref{D condition}) of $\theta$ and the inequality (\ref{compare}) again together with the fact that $0\leq\kappa<1$, we thus have
\begin{align}\label{theta4}
\sum_{l=1}^\infty(l+1)\cdot\frac{\theta(w(2^{l+1}B))}{\theta(w(B))}\cdot\frac{w(B)}{w(2^{l+1}B)}
&\leq C\sum_{l=1}^\infty(l+1)\cdot\frac{w(B)^{1-\kappa}}{w(2^{l+1}B)^{1-\kappa}}\notag\\
&\leq C\sum_{l=1}^\infty(l+1)\cdot\left(\frac{|B|}{|2^{l+1}B|}\right)^{\delta^*(1-\kappa)}\notag\\
&\leq C\sum_{l=1}^\infty(l+1)\cdot\left(\frac{1}{2^{(l+1)n}}\right)^{\delta^*(1-\kappa)}\notag\\
&\leq C.
\end{align}
Substituting the above inequality \eqref{theta4} into the term $K'_6$, we finally obtain
\begin{equation*}
K'_6\leq C\cdot\sup_B\left\{\frac{\Phi\Big(\frac{w(B)}{\theta(w(B))}\Big)}{w(B)}
\int_{B}\Phi\bigg(\frac{\|\vec{f}(z)\|_{\ell^2}}{\sigma}\bigg)\cdot w(z)\,dz\right\}.
\end{equation*}
Summing up all the above estimates, we therefore conclude the proof of the main theorem.
\end{proof}

\newtheorem*{remark}{Further remark}
\begin{remark}
Let $p=1$, $0\leq\kappa<1$, $\theta$ satisfy the $\mathcal D_\kappa$ condition $(\ref{D condition})$ and $w$ be a weight function on $\mathbb R^n$. We denote by $\mathcal M^{1,\theta}_{L\log L}(w)$ the generalized weighted Morrey space of $L\log L$ type, the space of all locally integrable functions $f$ defined on $\mathbb R^n$ with finite norm $\big\|f\big\|_{\mathcal M^{1,\theta}_{L\log L}(w)}$.
\begin{equation*}
\mathcal M^{1,\theta}_{L\log L}(w):=\left\{f\in L^1_{loc}(w):\big\|f\big\|_{\mathcal M^{1,\theta}_{L\log L}(w)}<\infty\right\},
\end{equation*}
where $\Phi(t)=t\cdot(1+\log^+t)$ and
\begin{equation*}
\big\|f\big\|_{\mathcal M^{1,\theta}_{L\log L}(w)}:=\sup_B\left\{\frac{\Phi\Big(\frac{w(B)}{\theta(w(B))}\Big)}{w(B)}\int_B\big|f(x)\big|w(x)\,dx\right\},
\end{equation*}
or
\begin{equation*}
\big\|f\big\|_{\mathcal M^{1,\theta}_{L\log L}(w)}
:=\sup_B\left\{\frac{1+\log^+\Big(\frac{w(B)}{\theta(w(B))}\Big)}{\theta(w(B))}\int_B\big|f(x)\big|w(x)\,dx\right\}.
\end{equation*}
Obviously, we have $\mathcal M^{1,\theta}(w)\supseteq\mathcal M^{1,\theta}_{L\log L}(w)$ by the definition. Then the corresponding estimate in Theorem \ref{mainthm:5} reads
\begin{equation*}
\begin{split}
&\frac{1}{\theta(w(B))}\cdot w\bigg(\bigg\{x\in B:\bigg(\sum_{j}\big|\big[b,\mathcal S_{\alpha}\big](f_j)(x)\big|^2\bigg)^{1/2}>\sigma\bigg\}\bigg)
\leq C\cdot\left\|\Phi\bigg(\frac{\|\vec{f}\|_{\ell^2}}{\sigma}\bigg)\right\|_{\mathcal M^{1,\theta}_{L\log L}(w)}.
\end{split}
\end{equation*}
Roughly speaking, we can say that the vector-valued commutator generated with BMO function is bounded from $\mathcal M^{1,\theta}_{L\log L}(w)$ to $W\mathcal M^{1,\theta}(w)$ from the above definitions. In comparison with the conclusions of Theorems \ref{mainthm:4} and \ref{mainthm:5}, it is natural to ask the question whether or not this vector-valued commutator has a more refined estimate:
\begin{equation*}
\begin{split}
&\frac{1}{\theta(w(B))}\cdot w\bigg(\bigg\{x\in B:\bigg(\sum_{j}\big|\big[b,\mathcal S_{\alpha}\big](f_j)(x)\big|^2\bigg)^{1/2}>\sigma\bigg\}\bigg)
\leq C\cdot\left\|\Phi\bigg(\frac{\|\vec{f}\|_{\ell^2}}{\sigma}\bigg)\right\|_{\mathcal M^{1,\theta}(w)}.
\end{split}
\end{equation*}
By the technique used in this article, there is an ``extra" $\log^+(\cdot)$ term in the estimation of $K'_5$. Whether this term can be removed is not known in general.
\end{remark}

\end{document}